\documentclass[a4paper, intlimits, reqno]{amsart}

\usepackage[english]{babel}
\usepackage[T1]{fontenc}
\usepackage[utf8]{inputenc}

\usepackage{amsmath}
\usepackage{amssymb}
\usepackage{MnSymbol}
\usepackage{amsthm}
\allowdisplaybreaks
\usepackage{amsfonts}
\usepackage{mathrsfs} 
\usepackage{mathtools}
\usepackage[all]{xy}
\usepackage{nicefrac}
\usepackage{enumitem}

\usepackage{multicol}
\usepackage{url}
\usepackage{dsfont}
\usepackage[numbers,sort&compress]{natbib}
\usepackage{doi}
\usepackage{prettyref}
\usepackage{xcolor}
\usepackage{orcidlink}

\newrefformat{defn}{Definition \ref{#1}}
\newrefformat{rem}{Remark \ref{#1}}
\newrefformat{sect}{Section \ref{#1}}
\newrefformat{sub}{Section \ref{#1}}
\newrefformat{prop}{Proposition \ref{#1}}
\newrefformat{thm}{Theorem \ref{#1}}
\newrefformat{cor}{Corollary \ref{#1}}
\newrefformat{ex}{Example \ref{#1}}

\swapnumbers
\newtheoremstyle{dotless}{}{}{\itshape}{}{\bfseries}{}{}{}
\theoremstyle{dotless}
\theoremstyle{plain}
\newtheorem{thm}{Theorem}[section]

\newtheorem{prop}[thm]{Proposition}
\newtheorem{cor}[thm]{Corollary}
\theoremstyle{definition}
\newtheorem{defn}[thm]{Definition}
\newtheorem{rem}[thm]{Remark}

\newcommand{\N} {\mathbb{N}}

\newcommand{\R} {\mathbb{R}}
\newcommand{\C} {\mathbb{C}}
\newcommand{\K} {\mathbb{K}}
\newcommand{\D} {\mathbb{D}}
\newcommand{\F} {\mathcal{F}(\Omega)}
\newcommand{\FE} {\mathcal{F}(\Omega,E)}
\newcommand{\G} {\mathcal{G}(\Omega)}
\newcommand{\GE} {\mathcal{G}(\Omega,E)}
\newcommand{\FV} {\mathcal{F}v(\Omega)}
\newcommand{\FVE} {\mathcal{F}v(\Omega,E)}
\newcommand{\acx} {\operatorname{acx}}
\newcommand{\oacx} {\overline{\operatorname{acx}}}

\DeclareMathOperator{\id}{id}
\providecommand{\differential}{\mathrm{d}}
\renewcommand{\d}{\differential}
\newcommand{\euler}{\mathrm{e}}

\newcommand{\vertiii}[1]{{\left\vert\kern-0.25ex\left\vert\kern-0.25ex\left\vert #1 
    \right\vert\kern-0.25ex\right\vert\kern-0.25ex\right\vert}}

\makeatletter
\newcommand{\fakephantomsection}{%
  \Hy@GlobalStepCount\Hy@linkcounter%
  \Hy@MakeCurrentHref{\@currenvir.\the\Hy@linkcounter}
  \Hy@raisedlink{\hyper@anchorstart{\@currentHref}\hyper@anchorend}%
}
\makeatother

\begin{document}

\title[Mixed topologies on Saks spaces of vector-valued functions]{Mixed topologies on Saks spaces of vector-valued functions}
\author[K.~Kruse]{Karsten Kruse\,\orcidlink{0000-0003-1864-4915}}
\address[Karsten Kruse]{University of Twente, Department of Applied Mathematics, P.O. Box 217, 7500 AE Enschede, The Netherlands, and Hamburg University of Technology, Institute of Mathematics, Am Schwarzenberg-Campus~3, 21073 Hamburg, Germany}

\email{k.kruse@utwente.nl}

\thanks{K.~Kruse acknowledges the support by the Deutsche Forschungsgemeinschaft (DFG) within the Research Training
 Group GRK 2583 ``Modeling, Simulation and Optimization of Fluid Dynamic Applications''.}



\date{\today}
\begin{abstract}
We study Saks spaces of functions with values in a normed space and the associated mixed topologies. 
We are interested in properties of such Saks spaces and mixed topologies which are relevant for applications in the 
theory of bi-continuous semigroups. 
In particular, we are interested if such Saks spaces are complete, semi-Montel, C-sequential or a (strong) Mackey space 
with respect to the mixed topology. Further, we consider the question whether the mixed and the submixed topology coincide on 
such Saks spaces and seek for explicit systems of seminorms that generate the mixed topology. 
\end{abstract}
\maketitle

\section{Introduction}\label{sect:intro}

This paper is dedicated to Saks spaces of vector-valued functions and their properties. A \emph{Saks space} is a triple 
$(X,\|\cdot\|,\tau)$ consisting of a normed space $(X,\|\cdot\|)$ and a coarser locally convex Hausdorff topology $\tau$ on $X$ 
such that the norm $\|\cdot\|$ is the supremum taken over some directed system of continuous seminorms that generates the 
$\tau$-topology, see \cite{cooper1978}. Associated to a Saks space is the \emph{mixed topology} 
$\gamma\coloneqq\gamma(\|\cdot\|,\tau)$, which was introduced in \cite{wiweger1961} and is the finest locally convex Hausdorff, 
even linear, topology between the $\|\cdot\|$-topology and $\tau$. Sequentially complete Saks spaces, i.e.~$(X,\gamma)$ 
is sequentially complete, are needed in the theory of bi-continuous semigroups, 
which were introduced in \cite{kuehnemund2001,kuehnemund2003}, to treat semigroups on Banach spaces $(X,\|\cdot\|)$ which are usually 
not strongly continuous w.r.t.~the norm $\|\cdot\|$ but only strongly continuous w.r.t.~the coarser topology $\tau$, 
e.g.~dual semigroups, implemented semigroups or transition semigroups like the Ornstein--Uhlenbeck semigroup on the space of bounded continuous functions on a Polish space. 

Besides sequential completeness there are several other properties of Saks spaces that are of importance in applications. 
The Lumer--Phillips generation theorems for bi-continuous semigroups from \cite{kruse_seifert2022b} need knowledge of explicit 
systems of seminorms that generate the the mixed topology $\gamma$ because the concept of dissipativity depends on the choice of the system of seminorms. There is another locally convex Hausdorff topology associated to a Saks space, namely the \emph{submixed topology} 
$\gamma_{\operatorname{s}}\coloneqq\gamma_{\operatorname{s}}(\|\cdot\|,\tau)$, which is defined by an explicit system of seminorms and is in general coarser than $\gamma$ but has the same convergent sequences as $\gamma$, see \prettyref{defn:mixed_top_Saks}. 
Therefore one is interested in the question when $\gamma$ and $\gamma_{\operatorname{s}}$ coincide. Moreover, the generation theorems 
like \cite[Theorems 3.10, 3.17, Corollary 3.15]{kruse_seifert2022b} need the completeness of the Saks space, 
i.e.~that $(X,\gamma)$ is complete. A sufficient conditions for $\gamma=\gamma_{\operatorname{s}}$ is that $(X,\gamma)$ is a semi-Montel space, which also implies that $(X,\gamma)$ is a complete semi-reflexive space and semi-reflexivity is needed for 
\cite[Theorems 3.17]{kruse_seifert2022b} as well. 
On the other hand, the Lumer--Phillips generation theorem \cite[Theorem 3.15, p.~75]{budde_wegner2022} for bi-continuous semigroups  needs that $(X,\gamma_{\operatorname{s}})$ is complete (see \cite[Theorem 3.11, Remark 3.12 (b)]{kruse_seifert2022b}).

The question whether $\gamma$ and $\gamma_{\operatorname{s}}$ coincide is also important for perturbation results of 
bi-continuous semigroups. If $\gamma=\gamma_{\operatorname{s}}$ and the Saks space is sequentially complete and \emph{C-sequential}, 
i.e.~every convex sequentially open subset of $(X,\gamma)$ is already open, then a bi-continuous semigroup 
on the corresponding Saks space is already locally, even quasi-, equitight by 
\cite[Theorem 3.17 (b), p.~13]{kruse_schwenninger2022}. Locally equitight bi-continuous semigroups are sometimes just called ``tight'' or ``local'' (see \cite{farkas2003,es_sarhir2006}) and local equitightness is needed for perturbation theorems like 
\cite[Theorem 1.2, p.~669]{es_sarhir2006}, \cite[Theorems 2.4, 3.2, p.~92, 94--95]{farkas2004}, \cite[Remark 4.1, p.~101]{farkas2004},
\cite[Theorem 5, p.~8]{budde2021} and \cite[Theorem 3.3, p.~582]{budde2021a}. Equitightness is relevant in ergodic theory for 
bi-continuous semigroups, see \cite[Remark 3.5 (ii), p.~147, Proposition 3.8, p.~150]{albanese2011}.

Apart from its relation to local equitightness its is also known that every bi-continuous semigroup on a sequentially complete 
C-sequential Saks space is locally, even quasi-, $\gamma$-equicontinuous by \cite[Theorem 7.4, p.~180]{kraaij2016} and 
\cite[Theorem 3.17 (a), p.~13]{kruse_schwenninger2022}. Equicontinuity and local equicontinuity are needed for 
perturbation results like dissipative perturbations or Desch--Schappacher perturbations \cite{albanese2016,jacob2015} 
and the infinitesimal description of Markov processes \cite{goldys2022}. 
Sequentially complete C-sequential Saks spaces also play a role in the duality between cost-uniform approximate 
null-controllability and final state observability, see \cite[Theorem 5.18, p.~441]{kruse_seifert2022a}. 
A sufficient condition for $(X,\gamma)$ being C-sequential is that $(X,\gamma)$ is a Mackey--Mazur space 
by \cite[Corollary 7.6, p.~52]{wilansky1981}. Here, $(X,\gamma)$ being a \emph{Mackey space} means that $\gamma$ 
is the Mackey topology of a dual pairing $\langle X,Y\rangle$ where $Y$ is a Banach space topologically isomorphic 
to the strong dual $(X,\gamma)_{b}'$, and being a \emph{Mazur space} means that all sequentially $\gamma$-continuous linear functionals are already $\gamma$-continuous. The question whether $(X,\gamma)$ is a Mackey space or even a \emph{strong} Mackey space, 
i.e.~a Mackey space such that $\sigma(Y,X)$-compact subsets of $Y$ are $\gamma$-equicontinuous, is interesting in itself, see 
e.g.~\cite[p.~553]{kunze2009} and \cite[Propositions 3.4, 4.9, p.~161, 166]{kraaij2016}. 
The condition that $(X,\gamma)$ is a sequentially complete Mackey--Mazur space is also sufficient for 
the existence of a dual bi-continuous semigroup of a bi-continuous semigroup in the sun dual theory for bi-continuous semigroups, 
see \cite[3.8 Theorem (b), p.~9--10]{kruse_schwenninger2023}.

We are interested in all of the properties listed above in the case of Saks spaces of vector-valued functions. 
Let us give an outline of our paper. In \prettyref{sect:notions} we briefly recall some notions and results from 
the theory of Saks spaces and give a characterisation of the approximation property of $(X,\gamma)$ in the case that 
$(X,\gamma)$ is a semi-Montel space in \prettyref{prop:approx_prop}. 

In \prettyref{sect:linearisation_weak} we start with a Saks space $(\F,\|\cdot\|,\tau)$ of real- or complex-valued functions 
on a non-empty set $\Omega$ such that $(\F,\|\cdot\|)$ is a Banach space and $\gamma=\gamma_{\operatorname{s}}$. 
We construct a weak $E$-valued version $(\FE_{\sigma},\|\cdot\|_{\sigma}^{E},\tau_{\sigma}^{E})$ of this space in a canonical way, 
where $E$ is a normed space (or more general a locally convex Hausdorff space), and show in 
\prettyref{thm:linearisation_full_mixed} that this triple is a complete Saks space and even complete when equipped with the 
submixed topology if $(\F,\|\cdot\|,\tau)$ is semi-Montel w.r.t.~$\gamma$, $\tau$ finer than the topology of pointwise convergence and $E$ a 
Banach space. The proof of this result is based on linearisation via Schwartz' $\varepsilon$-product and as a byproduct 
we also get a characterisation of $(\F,\gamma)$ having the approximation property in \prettyref{cor:approx_prop_vector_valued}.
We apply this result to weak $E$-valued versions of the Hardy space, weighted Bergman space and the Dirichlet space, 
whose properties we collect in \prettyref{cor:hardy_bergman_dirichlet}

In \prettyref{sect:strong_Saks_function} we consider a different way of defining an $E$-valued version of 
$(\F,\|\cdot\|,\tau)$ in \prettyref{defn:standard_space} which is available for some spaces and often stronger in the sense that 
is a subspace of $\FE_{\sigma}$ and sometimes even a strict subspace, see \prettyref{prop:relation_weak_strong}. 
In \prettyref{thm:standard_saks_space} we collect some of properties we are interested in of such strong $E$-valued Saks 
function spaces. Then we turn to specific examples. Among them are weighted spaces of continuous functions in \prettyref{cor:cv_mixed}, 
weighted space of holomorphic functions in \prettyref{cor:subspace_cont_holo_mixed}, 
weighted kernels of hypoelliptic linear partial differential operators in spaces of smooth functions in \prettyref{cor:subspace_cont_hypo_mixed}, 
in particular weighted spaces of harmonic functions, weighted Bloch spaces in \prettyref{cor:bloch_mixed}, 
spaces of Lipschitz continuous functions in \prettyref{cor:lipschitz_vanish_mixed} and spaces of $k$-times 
continuously partially differentiable functions on some open bounded set $\Omega\subset\R^{d}$ whose partial derivatives 
extend continuously to the boundary of $\Omega$ and 
whose partial derivatives of order $k$ are $\alpha$-H\"older continuous for some $0<\alpha\leq 1$ 
in \prettyref{cor:hoelder_diff__mixed}.

We close our paper with \prettyref{sect:dual_submixed} where we characterise the dual of the spaces from the preceding section 
with respect to a certain submixed topology which sometimes coincides with the mixed topology, 
see \prettyref{thm:dual_submixed_finite}. The interest in such a characterisation may also be motivated by the 
Lumer--Phillips generation theorem \cite[Corollary 3.15]{kruse_seifert2022b} which involves the dual w.r.t.~mixed topology.

\section{Notions and preliminaries}
\label{sect:notions}

In this short section we recall some basic notions from the theory of locally convex spaces, Saks spaces and 
mixed topologies. For a locally convex Hausdorff space $X$ over the field $\K\coloneqq\R$ or $\C$ we denote by $X'$ the topological linear dual space of $X$. If we want to emphasize the dependency on the locally convex Hausdorff topology $\tau$ of $X$, 
we write $(X,\tau)$ and $(X,\tau)'$ instead of just $X$ and $X'$, respectively. 
We write $(X,\tau)_{t}'$ for the space $(X,\tau)'$ equipped 
with the locally convex topology of uniform convergence on the bounded subsets of $(X,\tau)$ if $t=b$, 
and on the absolutely convex compact subsets of $(X,\tau)$ if $t=\kappa$. 
For two locally convex Hausdorff spaces $(X,\tau)$ and $(E,\tau_{E})$ we denote by $L((X,\tau),(E,\tau_{E}))$ the space 
of continuous linear maps from $(X,\tau)$ to $(E,\tau_{E})$. 
By \cite[Chap.~I, \S1, D\'{e}finition, p.~18]{schwartz1957} the \emph{$\varepsilon$-product} of Schwartz of 
$(X,\tau)$ and $(E,\tau_{E})$ is defined by 
$
X\varepsilon E\coloneqq L_{e}((X,\tau)_{\kappa}',(E,\tau_{E}))
$
where $L_{e}((X,\tau)_{\kappa}',(E,\tau_{E}))$ is the space $L((X,\tau)_{\kappa}',(E,\tau_{E}))$ 
equipped with the topology of uniform convergence on the equicontinuous subsets of $(X,\tau)'$. 
We identify the tensor product $X\otimes E$ with the linear finite rank operators in 
$X\varepsilon E$ and recap that $X$ has the \emph{approximation property} of Schwartz if and only if $X\otimes E$ is dense 
in $X\varepsilon E$ for all Banach spaces $E$ (see e.g.~\cite[Satz 10.17, p.~250]{kaballo2014}). 
Moreover, for two locally convex Hausdorff topologies $\tau_{0}$ and $\tau_{1}$ 
on $X$ we write $\tau_{0}\leq\tau_{1}$ if $\tau_{0}$ is coarser than $\tau_{1}$. 
We write $\tau_{\operatorname{co}}^{E}$ for the \emph{compact-open topology}, i.e.~the topology of uniform convergence on compact 
subsets of $\Omega$, on the space $\mathcal{C}(\Omega,E)$ of continuous functions on a topological Hausdorff space $\Omega$ with values  
in a locally convex Hausdorff space $E$. If $E=\K$, we just write $\tau_{\operatorname{co}}\coloneqq\tau_{\operatorname{co}}^{\K}$. 
In addition, we write $\tau_{\operatorname{p}}$ for the \emph{topology of pointwise convergence} on the space $\K^{\Omega}$ of $\K$-valued functions on a set $\Omega$. By a slight abuse of notation we also use the symbols $\tau_{\operatorname{co}}^{E}$ and 
$\tau_{\operatorname{p}}$ for the relative compact-open topology and the relative topology of pointwise convergence on topological subspaces of $\mathcal{C}(\Omega,E)$ and $\K^{\Omega}$, respectively (cf.~\cite[Section 2]{kruse2023a,kruse2023b}).

Let us recall the definition of the mixed topology, \cite[Section 2.1]{wiweger1961}, and the notion of a Saks space, 
\cite[I.3.2 Definition, p.~27--28]{cooper1978}, which will be important for the rest of the paper.

\begin{defn}[{\cite[Definition 2.2, p.~3]{kruse_schwenninger2022}}]\label{defn:mixed_top_Saks}
Let $(X,\|\cdot\|)$ be a normed space and $\tau$ a locally convex Hausdorff topology on $X$ such that $\tau\leq\tau_{\|\cdot\|}$ 
where $\tau_{\|\cdot\|}$ denotes the topology induced by $\|\cdot\|$.
Then
\begin{enumerate}
\item the \emph{mixed topology} $\gamma \coloneqq \gamma(\|\cdot\|,\tau)$ is
the finest linear topology on $X$ that coincides with $\tau$ on $\|\cdot\|$-bounded sets and such that 
$\tau\leq\gamma \leq \tau_{\|\cdot\|}$,
\item the triple $(X,\|\cdot\|,\tau)$ is called a \emph{Saks space} if there exists a directed system of continuous seminorms 
$\Gamma_{\tau}$ that generates the topology $\tau$ such that   
\begin{equation}\label{eq:saks}
\|x\|=\sup_{q\in\Gamma_{\tau}} q(x), \quad x\in X.
\end{equation}
\end{enumerate}
\end{defn}

In comparison to \cite[Definition 2.2, p.~3]{kruse_schwenninger2022} we dropped the assumption that the 
space $(X,\|\cdot\|)$ should be complete. 
The mixed topology $\gamma$ is actually Hausdorff locally convex and the definition given above is equivalent to the one 
introduced by Wiweger \cite[Section 2.1]{wiweger1961} due to \cite[Lemmas 2.2.1, 2.2.2, p.~51]{wiweger1961}. 

It is often useful to have a characterisation of the mixed topology by generating systems of continuous seminorms, 
e.g.~the definition of dissipativity in Lumer--Phillips generation theorems for bi-continuous semigroups 
from \cite{kruse_seifert2022b} mentioned in the introduction 
depends on the choice of the generating system of seminorms of the mixed topology. 
For that purpose we introduce the following auxiliary topology. 

\begin{defn}[{\cite[Definition 3.9, p.~9]{kruse_schwenninger2022}}]\label{defn:submixed_top}
Let $(X,\|\cdot\|,\tau)$ be a Saks space and $\Gamma_{\tau}$ a directed system of continuous seminorms 
that generates the topology $\tau$ and fulfils \eqref{eq:saks}. We set 
\[
\mathcal{N}\coloneqq \{(q_{n},a_{n})_{n\in\N}\;|\;(q_{n})_{n\in\N}\subset\Gamma_{\tau},\,(a_{n})_{n\in\N}\in c_{0}^{+}\}
\]
where $c_0^{+}$ is the family of all real non-negative null-sequences.
For $(q_{n},a_{n})_{n\in\N}\in\mathcal{N}$ we define the seminorm
\[
 \vertiii{x}_{(q_{n},a_{n})_{n\in\N}}\coloneqq\sup_{n\in\N}q_{n}(x)a_{n},\quad x\in X.
\]
We denote by $\gamma_{\operatorname{s}}\coloneqq\gamma_{\operatorname{s}}(\|\cdot\|,\tau)$ the locally convex Hausdorff topology 
that is generated by the system of seminorms $(\vertiii{\cdot}_{(q_n,a_n)_{n\in\N}})_{(q_n,a_n)_{n\in\N}\in\mathcal{N}}$ and 
call it the \emph{submixed topology}.
\end{defn}

We note that the submixed topology $\gamma_{\operatorname{s}}$ does not dpeend on the choice of $\Gamma_{\tau}$ 
that fufils \eqref{eq:saks}. By \cite[I.1.10 Proposition, p.~9]{cooper1978}, \cite[I.4.5 Proposition, p.~41--42]{cooper1978} and 
\cite[Lemma A.1.2, p.~72]{farkas2003} we have the following relation between the mixed and the submixed topology.

\begin{rem}[{\cite[Remark 3.10, p.~9]{kruse_schwenninger2022}}]\label{rem:mixed=submixed}
Let $(X,\|\cdot\|,\tau)$ be a Saks space, $\Gamma_{\tau}$ a directed system of continuous seminorms 
that generates the topology $\tau$ and fulfils \eqref{eq:saks}, $\gamma\coloneqq\gamma(\|\cdot\|,\tau)$ the mixed 
and $\gamma_{\operatorname{s}}\coloneqq\gamma_{\operatorname{s}}(\|\cdot\|,\tau)$ the submixed topology.
\begin{enumerate}
\item[(a)] We have $\tau\leq \gamma_{\operatorname{s}}\leq \gamma$ and $\gamma_{\operatorname{s}}$ has the same convergent 
sequences as $\gamma$.
\item[(b)] If 
 \begin{enumerate}
 \item[(i)] for every $x\in X$, $\varepsilon>0$ and $q\in\Gamma_{\tau}$ there are $y,z\in X$ such that $x=y+z$, 
 $q(z)=0$ and $\|y\|\leq q(x)+\varepsilon$, or 
 \item[(ii)] the $\|\cdot\|$-closed unit ball $B_{\|\cdot\|}\coloneqq\{x\in X\;|\; \|x\|\leq 1\}$ is $\tau$-compact,
 \end{enumerate}
then it holds $\gamma=\gamma_{\operatorname{s}}$. 
\end{enumerate}
\end{rem}

The submixed topology $\gamma_{\operatorname{s}}$ was originally introduced in \cite[Theorem 3.1.1, p.~62]{wiweger1961} where a 
proof of \prettyref{rem:mixed=submixed} (b) can be found, too. The following notions will also be needed, which were introduced 
in \cite[I.3.2 Definition, p.~27--28]{cooper1978}, \cite[3.3 Definition, p.~7]{kruse2023b} and 
\cite[Definition 2.2]{kruse_seifert2022b}.

\begin{defn}
Let $(X,\|\cdot\|,\tau)$ be a Saks space. 
\begin{enumerate}
\item We call $(X,\|\cdot\|,\tau)$ \emph{complete} if $(X,\gamma)$ is complete.
\item We call $(X,\|\cdot\|,\tau)$ \emph{semi-Montel} if $(X,\gamma)$ is a semi-Montel space.
\item We call $(X,\|\cdot\|,\tau)$ \emph{C-sequential} if $(X,\gamma)$ is a C-sequential space, i.e.~every convex 
sequentially open subset of $(X,\gamma)$ is already open (see \cite[p.~273]{snipes1973}).
\end{enumerate}
\end{defn}

\begin{rem}\label{rem:pre_Saks_unit_ball_char}
Let $(X,\|\cdot\|)$ be a normed space and $\tau$ a locally convex Hausdorff topology on $X$. 
Set $X^{\ast}\coloneqq (X,\|\cdot\|)'$ and denote by $\|\cdot\|_{X^{\ast}}$ the dual norm on $X^{\ast}$.
\begin{enumerate}
\item $(X,\|\cdot\|,\tau)$ is a semi-Montel Saks space if and only if $B_{\|\cdot\|}$ 
is $\tau$-compact by \cite[I.1.13 Proposition, p.~11]{cooper1978} and \cite[3.6 Remark (c), p.~8]{kruse2023b}.
\item If $(X,\|\cdot\|,\tau)$ is a semi-Montel Saks space, then $(X,\|\cdot\|,\tau)$ and $(X,\|\cdot\|)$ are complete by 
\cite[3.6 Remark (a), (b), p.~8]{kruse2023b}.
\item If $(X,\|\cdot\|,\tau)$ is a semi-Montel Saks space and $\tau_{0}$ a locally convex Hausdorff topology on $X$ such that 
$\tau_{0}\leq\tau$, then $(X,\|\cdot\|,\tau_{0})$ is a semi-Montel Saks space and 
$\gamma_{\operatorname{s}}(\|\cdot\|,\tau)=\gamma(\|\cdot\|,\tau)=\gamma(\|\cdot\|,\tau_{0})
=\gamma_{\operatorname{s}}(\|\cdot\|,\tau_{0})$ by part (a), condition (ii) of \prettyref{rem:mixed=submixed} (b) and 
\cite[3.6 Remark (c), p.~8]{kruse2023b}.
\item Let $(X,\|\cdot\|,\tau)$ be a Saks space, set $X_{\gamma}'\coloneqq (X,\gamma)'$ and denote by $\|\cdot\|_{X_{\gamma}'}$ 
the restriction of $\|\cdot\|_{X^{\ast}}$ to $X_{\gamma}'$. 
Then $(X,\gamma)_{b}'=(X_{\gamma}',\|\cdot\|_{X_{\gamma}'})$ and $(X_{\gamma}',\|\cdot\|_{X_{\gamma}'})$ is a Banach space by 
\cite[I.1.18 Proposition, p.~15]{cooper1978}.
\end{enumerate}
\end{rem}

We close this section with the following observation concerning the approximation property of $(X,\gamma)$ in the 
semi-Montel case, whose proof is an adaptation of \cite[Theorem 4.6 (i)$\Leftrightarrow$(ii), p.~651--652]{vargas2018} where 
$X=\operatorname{Lip}_{0}(\Omega)$ is the space of $\K$-valued Lipschitz continuous functions on a metric space $\Omega$ that vanish 
at the origin (see \prettyref{cor:lipschitz_vanish_mixed}).

\begin{prop}\label{prop:approx_prop}
Let $(X,\|\cdot\|,\tau)$ be a semi-Montel Saks space. Then the following assertions are equivalent.
\begin{enumerate}
\item $(X,\gamma)$ has the approximation property.
\item $(X_{\gamma}',\|\cdot\|_{X_{\gamma}'})$ has the approximation property.
\end{enumerate}
\end{prop}
\begin{proof}
(a)$\Rightarrow$(b): Due to \prettyref{rem:pre_Saks_unit_ball_char} (a) and 
\cite[I.4.1 Proposition, p.~38]{cooper1978} (or \cite[I.4.2 Corollary (d), p.~38]{cooper1978} in combination 
with \cite[Theorem 4.1, p.~43]{mujica1984}) we have $(X,\gamma)=(X_{\gamma}',\|\cdot\|_{X_{\gamma}'})_{\kappa}'$. 
Hence $E\coloneqq (X_{\gamma}',\|\cdot\|_{X_{\gamma}'})$ has the approximation property 
by \cite[Corollary 1.3, p.~144]{dineen2004} because $(X,\gamma)=E_{\kappa}'$ has the approximation property.

(b)$\Rightarrow$(a): We note that $(X_{\gamma}',\|\cdot\|_{X_{\gamma}'})=(X,\gamma)_{b}'=(X,\gamma)_{\kappa}'$ 
by \prettyref{rem:pre_Saks_unit_ball_char} (d) and the semi-Montel property of $(X,\gamma)$. 
Thus $E\coloneqq (X,\gamma)$ has the approximation property by \cite[Corollary 1.3, p.~144]{dineen2004} because 
$(X_{\gamma}',\|\cdot\|_{X_{\gamma}'})=E_{\kappa}'$ has the approximation property.
\end{proof}

In particular, the proof above shows that $(X,\gamma)
=(X_{\gamma}',\|\cdot\|_{X_{\gamma}'})_{\kappa}'=((X,\gamma)_{b}')_{\kappa}'$, i.e.~$(X,\gamma)$ 
is a \emph{DFC-space} by \cite[Theorem 4.1, p.~43]{mujica1984}, if $(X,\|\cdot\|,\tau)$ is a semi-Montel Saks space.
Sufficient conditions that guarantee that $(X,\gamma)$ has the approximation property may be found 
in \cite[I.4.20 Proposition, p.~53]{cooper1978} and \cite[I.4.21, I.4.22 Corollaries, p.~54]{cooper1978}. 

\section{Saks spaces of weak vector-valued functions}
\label{sect:linearisation_weak}

In this section we introduce Saks spaces of weak vector-valued functions. We use a linearisation based on the $\varepsilon$-product 
to show that they are complete w.r.t.~the mixed and the submixed topology if their scalar-valued version is semi-Montel, 
$\tau_{\operatorname{p}}\leq\tau$ and they have values in a Banach space. 

Let $(\F,\|\cdot\|)$ be a Banach space of $\K$-valued functions on a non-empty set $\Omega$ such that 
$\tau_{\operatorname{p}}\leq\tau_{\|\cdot\|}$. 
We recall from \cite[p.~31]{kruse2023a} a canonical construction of a weak vector-valued version of such a space.
For a locally convex Hausdorff space $E$ over $\K$ with directed system of seminorms $\Gamma_{E}$ generating its topology
we define the \emph{space of weak $E$-valued $\mathcal{F}$-functions} by 
\[
\FE_{\sigma}\coloneqq\{f\colon \Omega\to E\;|\;\forall\;e'\in E':\;e'\circ f\in\F\}.
\]
For $p\in\Gamma_{E}$ we set $U_{p}\coloneqq\{x\in E\;|\; p(x)<1\}$ and denote by $U_{p}^{\circ}$ the polar of $U_{p}$. 
Since $(\F,\|\cdot\|)$ is a Banach space, thus webbed, 
the supremum $\|f\|_{\sigma,p}\coloneqq \sup_{e'\in U_{p}^{\circ}}\|e'\circ f\|$ is finite 
for every $f\in\FE_{\sigma}$ and $p\in\Gamma_{E}$ by \cite[5.1 Remark, p.~31]{kruse2023a}. 
Hence the space $\FE_{\sigma}$ equipped with the system of seminorms $(\|\cdot\|_{\sigma,p})_{p\in\Gamma_{E}}$ is 
a locally convex Hausdorff space. If $(E,\|\cdot\|_{E})$ is a normed space with 
$\Gamma_{E}\coloneqq\{\|\cdot\|_{E}\}$, then $(\FE_{\sigma},\|\cdot\|_{\sigma}^{E})$ is a normed space where 
$\|\cdot\|_{\sigma}^{E}\coloneqq\|\cdot\|_{\sigma,\|\cdot\|_{E}}$. 

Now, let $\tau$ be an additional locally convex Hausdorff topology on $\F$ such that $(\F,\|\cdot\|,\tau)$ is a Saks space and 
$\gamma\coloneqq\gamma(\|\cdot\|,\tau)=\gamma_{\operatorname{s}}(\|\cdot\|,\tau)$. 
Then, by \prettyref{defn:submixed_top}, a directed system of seminorms that generates $\gamma$ is given by 
\[
 \vertiii{f}_{(q_{n},a_{n})_{n\in\N}}\coloneqq\sup_{n\in\N}q_{n}(f)a_{n},\quad f\in \F,
\]
for $(q_{n})_{n\in\N}\subset\Gamma_{\tau}$, 
where $\Gamma_{\tau}$ is a directed system of continuous seminorms 
that generates the topology $\tau$ and fulfils \eqref{eq:saks}, 
and $(a_{n})_{n\in\N}\in c_{0}^{+}$. We set 
\begin{equation}\label{eq:weak_submixed}
 \vertiii{f}_{\sigma,(q_{n},a_{n})_{n\in\N},p}
 \coloneqq\sup_{e'\in U_{p}^{\circ}}\vertiii{e'\circ f}_{(q_{n},a_{n})_{n\in\N}}
 =\sup_{e'\in U_{p}^{\circ}}\sup_{n\in\N}q_{n}(e'\circ f)a_{n},\;\; f\in \FE_{\sigma},
\end{equation}
for $p\in\Gamma_{E}$, $(q_{n})_{n\in\N}\subset\Gamma_{\tau}$ and $(a_{n})_{n\in\N}\in c_{0}^{+}$. 
Then for $p\in\Gamma_{E}$, $(q_{n})_{n\in\N}\subset\Gamma_{\tau}$ and 
$(a_{n})_{n\in\N}\in c_{0}^{+}$ it holds
\[
\vertiii{f}_{\sigma,(q_{n},a_{n})_{n\in\N},p}\leq \sup_{n\in\N}|a_{n}|\|f\|_{\sigma,p}<\infty
\]
for all $f\in \FE_{\sigma}$. So the system of seminorms 
$(\vertiii{f}_{\sigma,(q_{n},a_{n})_{n\in\N},p})_{(q_{n},a_{n})_{n\in\N}\in\mathcal{N},p\in\Gamma_{E}}$ induces a locally convex 
Hausdorff topology on $\FE_{\sigma}$ which we denote by $\gamma_{\sigma,\operatorname{s}}^{E}$.

\begin{rem}\label{rem:weak_vector_valued_ind_of_choice_seminorms}
Let $(\F,\|\cdot\|,\tau)$ be a Saks space of $\K$-valued functions on a non-empty set $\Omega$ such that $(\F,\|\cdot\|)$ is 
a Banach space, $\tau_{\operatorname{p}}\leq\tau_{\|\cdot\|}$ and $\gamma(\|\cdot\|,\tau)=\gamma_{\operatorname{s}}(\|\cdot\|,\tau)$, and $E$ a locally convex Hausdorff space over 
$\K$ with directed system of seminorms $\Gamma_{E}$ generating its topology. 
Then the topology $\gamma_{\sigma,\operatorname{s}}^{E}$ does not depend 
on the choice of the system of seminorms that generates $\gamma\coloneqq \gamma(\|\cdot\|,\tau)$. 
Indeed, let $\Gamma_{\gamma}$ be a another system of seminorms that generates $\gamma$. 
Then for every $(q_{n})_{n\in\N}\subset\Gamma_{\tau}$, $\Gamma_{\tau}$ as above, and $(a_{n})_{n\in\N}\in c_{0}^{+}$ there are 
$C_{0}\geq 0$ and $r_{0}\in\Gamma_{\gamma}$ such that 
$
\vertiii{f}_{\sigma,(q_{n},a_{n})_{n\in\N}}\leq C_{0}r_{0}(f)
$
for all $f\in\F$. On the other hand, for every and $r_{1}\in\Gamma_{\gamma}$ there are $C_{1}\geq 0$, 
$(\widetilde{q}_{n})_{n\in\N}\subset\Gamma_{\tau}$ and $(\widetilde{a}_{n})_{n\in\N}\in c_{0}^{+}$ such that 
$
 r_{1}(f)\leq C_{1}\vertiii{f}_{\sigma,(\widetilde{q}_{n},\widetilde{a}_{n})_{n\in\N}}
$
for all $f\in\F$. This implies that 
\[
\vertiii{f}_{\sigma,(q_{n},a_{n})_{n\in\N},p}\leq C_{0}\sup_{e'\in U_{p}^{\circ}}r_{0}(e'\circ f)
\]
and 
\[
\sup_{e'\in U_{p}^{\circ}}r_{1}(e'\circ f)\leq C_{1}\vertiii{f}_{\sigma,(\widetilde{q}_{n},\widetilde{a}_{n})_{n\in\N},p}
\]
for all $f\in\FE_{\sigma}$ and $p\in\Gamma_{E}$. Thus the system of seminorms given by 
\[
|f|_{\sigma,r,p}\coloneqq \sup_{e'\in U_{p}^{\circ}}r(e'\circ f),\quad f\in \FE_{\sigma},
\]
for $r\in\Gamma_{\gamma}$ and $p\in\Gamma_{E}$ also generates $\gamma_{\sigma,\operatorname{s}}^{E}$. 
Similarly, $\gamma_{\sigma,\operatorname{s}}^{E}$ does not depend on the choice of $\Gamma_{E}$.
\end{rem}

\begin{rem}\label{rem:weak_vector_valued_Saks}
Let $(\F,\|\cdot\|,\tau)$ be a Saks space of $\K$-valued functions on a non-empty set $\Omega$ such that $(\F,\|\cdot\|)$ is 
a Banach space, $\tau_{\operatorname{p}}\leq\tau_{\|\cdot\|}$ and $\gamma(\|\cdot\|,\tau)=\gamma_{\operatorname{s}}(\|\cdot\|,\tau)$, and $(E,\|\cdot\|_{E})$ a normed space over 
$\K$. Then $(\FE_{\sigma},\|\cdot\|_{\sigma}^{E},\tau_{\sigma}^{E})$ is a Saks space where 
$\tau_{\sigma}^{E}$ is the locally convex Hausdorff topology on $\FE_{\sigma}$ generated by the system of seminorms given by
\[
q_{\sigma}^{E}(f)\coloneqq\sup_{e'\in U_{\|\cdot\|_{E}}^{\circ}}q(e'\circ f)
=\sup_{e^{\ast}\in B_{\|\cdot\|_{E^{\ast}}}}q(e^{\ast}\circ f),\quad f\in\FE_{\sigma}, 
\]
for $q\in\Gamma_{\tau}$ and $\Gamma_{\tau}$ as above. Indeed, this follows from \prettyref{defn:mixed_top_Saks} 
and the observation
\[
\sup_{q\in\Gamma_{\tau}}q_{\sigma}^{E}(f)
=\sup_{q\in\Gamma_{\tau}}\sup_{e^{\ast}\in B_{\|\cdot\|_{E^{\ast}}}}q(e^{\ast}\circ f)
=\sup_{e^{\ast}\in B_{\|\cdot\|_{E^{\ast}}}}\sup_{q\in\Gamma_{\tau}}q(e^{\ast}\circ f)
=\sup_{e^{\ast}\in B_{\|\cdot\|_{E^{\ast}}}}\|e^{\ast}\circ f\|
=\|f\|_{\sigma}^{E}.
\]
Further, $\gamma_{\sigma,\operatorname{s}}^{E}=\gamma_{\operatorname{s}}(\|\cdot\|_{\sigma}^{E},\tau_{\sigma}^{E})$ 
by \prettyref{defn:submixed_top} and the definitions of $\tau_{\sigma}^{E}$ and $\gamma_{\sigma,\operatorname{s}}^{E}$.
\end{rem}

For a linear space $\F$ of $\K$-valued functions on a non-empty set $\Omega$ and $x\in\Omega$ we define the linear functional 
$\Delta(x)\colon\F\to\K$, $\Delta(x)(f)\coloneqq f(x)$. 

\begin{thm}\label{thm:linearisation_full_mixed}
Let $(\F,\|\cdot\|,\tau)$ be a semi-Montel Saks space of $\K$-valued functions on a non-empty set $\Omega$ such that 
$\tau_{\operatorname{p}}\leq\tau$, $E$ a complete locally convex Hausdorff space over $\K$ and 
set $\F_{\gamma}\coloneqq (\F,\gamma)$. Then the map 
\[
\chi\colon \F_{\gamma}\varepsilon E \to(\FE_{\sigma},\gamma_{\operatorname{s},\sigma}^{E}),\;\chi(u)\coloneqq u\circ \Delta = [x\mapsto u(\Delta(x))],
\]
is a topological isomorphism. In particular, $(\FE_{\sigma},\gamma_{\sigma,\operatorname{s}}^{E})$ is complete. 
If $E$ is a Banach space, then $(\FE_{\sigma},\|\cdot\|_{\sigma}^{E},\tau_{\sigma}^{E})$ is a complete Saks space.
\end{thm}
\begin{proof}
First, we show that $\chi$ is well-defined and linear.  
Since $\tau_{\operatorname{p}}\leq\tau\leq\gamma$, it holds $\Delta(x)\in\F_{\gamma}'$ for all $x\in\Omega$ 
(cf.~\cite[4.2 Remark, p.~12]{kruse2023b}). 
We note that $(\F_{\gamma})_{\kappa}'=(\F_{\gamma})_{b}'$ because $\F_{\gamma}$ is a semi-Montel space. 
Hence $\F_{\gamma}\varepsilon E =L_{e}((\F_{\gamma})_{b}',(E,\tau_{E}))$ where 
$\tau_{E}$ denotes the locally convex Hausdorff topology of $E$.
Let $u\in \F_{\gamma}\varepsilon E$ and $e'\in E'$.
Using that $e'\circ u\in ((\F_{\gamma})_{b}')'$ and the semi-reflexivity of the semi-Montel space $\F_{\gamma}$, 
we note that there is 
$f_{e'\circ u}\in\F$ such that $(e'\circ u)(f')=f'(f_{e'\circ u})$ for all $f'\in\F_{\gamma}'$. 
This implies with $f'=\Delta(x)$ that $(e'\circ u)\circ\Delta=f_{e'\circ u}$. 
Thus the map $\chi$ is well-defined and it is easily seen to be linear as well. 

Second, we show that $\chi$ is injective and continuous. 
Let $\Gamma_{\tau}$ be a directed system of continuous seminorms 
that generates the topology $\tau$ and fulfils \eqref{eq:saks}. 
For $(q_{n})_{n\in\N}\subset\Gamma_{\tau}$ and $(a_{n})_{n\in\N}\in c_{0}^{+}$ we define 
$U_{q_{n}a_{n}}\coloneqq \{f\in\F\;|\;q_{n}(f)a_{n}< 1\}$ and note that the sets 
\[
V_{(q_{n},a_{n})_{n\in\N}}\coloneqq \bigcap_{n\in\N}U_{q_{n}a_{n}}
\]
form a base of $\gamma$-neighbourhoods of zero by \prettyref{defn:submixed_top} since 
$\gamma=\gamma_{\operatorname{s}}$ by condition (ii) of \prettyref{rem:mixed=submixed} (b) 
and \prettyref{rem:pre_Saks_unit_ball_char}. 
By the bipolar theorem we have 
\[
V_{(q_{n},a_{n})_{n\in\N}}^{\circ}=
\oacx \Bigl(\bigcup_{n\in\N}U_{q_{n}a_{n}}^{\circ}\Bigr)
\eqqcolon \oacx(W_{(q_{n},a_{n})_{n\in\N}})
\]
where $\oacx(W_{(q_{n},a_{n})_{n\in\N}})$ denotes the closure in $(\F,\gamma)_{\kappa}'$ 
of the absolutely convex hull $\acx(W_{(q_{n},a_{n})_{n\in\N}})$ 
of $W_{(q_{n},a_{n})_{n\in\N}}\coloneqq\bigcup_{n\in\N}U_{q_{n}a_{n}}^{\circ}$ 
(see \cite[8.2.4 Corollary, p.~149]{jarchow1981}). Due to \cite[8.4, p.~152, 8.5, p.~156--157]{jarchow1981} the topology of 
$\F_{\gamma}\varepsilon E$ is generated by the seminorms 
\[
|u|_{(q_{n},a_{n})_{n\in\N},p}\coloneqq \sup_{y\in V_{(q_{n},a_{n})_{n\in\N}}^{\circ}}p(u(y))
=\sup_{y\in \oacx(W_{(q_{n},a_{n})_{n\in\N}})}p(u(y)),
\quad u\in \F_{\gamma}\varepsilon E,
\]
for $(q_{n})_{n\in\N}\subset\Gamma_{\tau}$, $(a_{n})_{n\in\N}\in c_{0}^{+}$ and $p\in\Gamma_{E}$ where $\Gamma_{E}$ 
denotes a system of seminorms that generates $\tau_{E}$. By the continuity of $u\in \F_{\gamma}\varepsilon E$ we have 
\[
|u|_{(q_{n},a_{n})_{n\in\N},p}=\sup_{y\in \acx(W_{(q_{n},a_{n})_{n\in\N}})}p(u(y)) 
\geq \sup_{y\in W_{(q_{n},a_{n})_{n\in\N}}}p(u(y)).
\]
On the other hand, for $y\in\acx(W_{(q_{n},a_{n})_{n\in\N}})$ there are 
$m\in\N$, $\lambda_{k}\in\K$, $f_{k}'\in U_{q_{k}a_{k}}^{\circ}$, $1\leq k\leq m$ 
with $\sum_{k=1}^{m}|\lambda_{k}|=1$ such that $y=\sum_{k=1}^{m}\lambda_{k}f_{k}'$. 
It follows that for all $u\in \F_{\gamma}\varepsilon E$
\[
  p(u(y))
\leq\sum_{k=1}^{m}|\lambda_{k}|p(u(f_{k}'))
\leq \sup_{1\leq k\leq m} p(u(f_{k}'))
\leq \sup_{z\in W_{(q_{n},a_{n})_{n\in\N}}}p(u(z))
\] 
and we deduce 
\begin{equation}\label{eq:seminorms_eps}
 |u|_{(q_{n},a_{n})_{n\in\N},p}
=\sup_{y\in W_{(q_{n},a_{n})_{n\in\N}}}p(u(y))
=\sup_{n\in\N}\sup_{f'\in U_{q_{n}a_{n}}^{\circ}}p(u(f')).
\end{equation}
By the first part of the proof there is $f_{e'\circ u}\in\F$ such that $(e'\circ u)(f')=f'(f_{e'\circ u})$ 
and $(e'\circ u)\circ\Delta=f_{e'\circ u}$ for all $u\in \F_{\gamma}\varepsilon E$, $f'\in\F_{\gamma}'$ 
and $e'\in E'$. We conclude that 
\begin{align}\label{eq:chi_cont}
 \vertiii{\chi(u)}_{\sigma,(q_{n},a_{n})_{n\in\N},p}\nonumber
&=\sup_{e'\in U_{p}^{\circ}}\sup_{n\in\N}q_{n}(e'\circ (u\circ\Delta))a_{n}
 =\sup_{e'\in U_{p}^{\circ}}\sup_{n\in\N}\sup_{f'\in U_{q_{n}a_{n}}^{\circ}}|f'(f_{e'\circ u})|\nonumber\\
&=\sup_{e'\in U_{p}^{\circ}}\sup_{n\in\N}\sup_{f'\in U_{q_{n}a_{n}}^{\circ}}|(e'\circ u)(f')|
 =\sup_{n\in\N}\sup_{f'\in U_{q_{n}a_{n}}^{\circ}}p(u(f'))\nonumber\\
&\underset{\mathclap{\eqref{eq:seminorms_eps}}}{=}|u|_{(q_{n},a_{n})_{n\in\N},p}
\end{align}
for all $u\in\F_{\gamma}\varepsilon E$. Hence $\chi$ is injective and continuous.  

Third, we show that $\chi$ is surjective and note that \eqref{eq:chi_cont} implies that the inverse of $\chi$ 
is continuous. Due to \prettyref{rem:pre_Saks_unit_ball_char} (d) we have
\[
\F_{\gamma}\varepsilon E =L_{e}((\F_{\gamma})_{b}',(E,\tau_{E}))=L_{e}((\F_{\gamma}',\|\cdot\|_{\F_{\gamma}'}),(E,\tau_{E})).
\]
Hence the surjectivity of $\chi$ is a consequence of \cite[5.5 Theorem, p.~33]{kruse2023a}, 
and \cite[4.1 Corollary, p.~12]{kruse2023b}.

Fourth, since $\F_{\gamma}$ and $E$ are complete, $\F_{\gamma}\varepsilon E$ is also complete 
by \cite[Satz 10.3, p.~234]{kaballo2014}, implying the completeness of $(\FE_{\sigma},\gamma_{\sigma,\operatorname{s}}^{E})$. 
Let $E$ be a Banach space. Due to \prettyref{rem:weak_vector_valued_Saks} and \prettyref{rem:mixed=submixed} (a) 
$(\FE_{\sigma},\|\cdot\|_{\sigma}^{E},\tau_{\sigma}^{E})$ is a Saks space 
and $\gamma_{\sigma,\operatorname{s}}^{E}=\gamma_{\operatorname{s}}(\|\cdot\|_{\sigma}^{E},\tau_{\sigma}^{E})
\leq \gamma(\|\cdot\|_{\sigma}^{E},\tau_{\sigma}^{E})$, which implies the completeness of the Saks space.
\end{proof}

The proof of the continuity of $\chi$ and its inverse in \prettyref{thm:linearisation_full_mixed} 
is similar to the proof of \cite[Lemma 7, p.~1517]{kruse2017}. Moreover, 
\prettyref{thm:linearisation_full_mixed} in combination with 
\prettyref{rem:weak_vector_valued_Saks}, \prettyref{prop:relation_weak_strong} (b) 
and \prettyref{cor:lipschitz_vanish_mixed} (a) and (d)
generalises \cite[Theorem 4.4, p.~648]{vargas2018} 
where $\FE=\operatorname{Lip}_{0}(\Omega,E)$, $E$ is a Banach space and 
$\gamma_{\operatorname{s},\sigma}^{E}=\gamma_{\operatorname{s}}(\|\cdot\|_{\operatorname{Lip}_{0}(\Omega,E)},\tau_{\mathcal{N}_{\Omega_{\operatorname{wd}}}}^{E})=\gamma\tau_{\gamma}$.
Due to \prettyref{prop:relation_weak_strong} (d) in combination with \prettyref{cor:cv_mixed} (b), \prettyref{cor:subspace_cont_holo_mixed} (d) and (e), 
\prettyref{cor:subspace_cont_hypo_mixed} (b) and (c), the result of \prettyref{thm:linearisation_full_mixed} 
is already contained in \cite[3.1 Bemerkung, p.~141]{bierstedt1973b} (cf.~\cite[5.2.10 Proposition, p.~77]{kruse2023}, 
\cite[5.2.17 Corollary, p.~80]{kruse2023} and 
\cite[4.8 Theorem, p.~878]{mujica1991}) for the weighted space $\FE=\mathcal{C}v(\Omega,E)$ of continuous functions from \prettyref{cor:cv_mixed} if $\Omega$ is discrete, the weighted space $\FE=\mathcal{H}v(\Omega,E)$ of holomorphic functions 
from \prettyref{cor:subspace_cont_holo_mixed} and the weighted kernel
$\FE=\mathcal{C}_{P}v(\Omega,E)$ of a hypoelliptic linear partial differential operator 
from \prettyref{cor:subspace_cont_hypo_mixed} even for quasi-complete locally convex Hausdorff spaces $E$. 
However, the proof is different. \prettyref{thm:linearisation_full_mixed} also allows us to characterise $(\F,\gamma)$ 
having the approximation property by approximation in $(\FE_{\sigma},\gamma_{\sigma,\operatorname{s}}^{E})$.

\begin{cor}\label{cor:approx_prop_vector_valued}
Let $(\F,\|\cdot\|,\tau)$ be a semi-Montel Saks space of $\K$-valued functions on a non-empty set $\Omega$ 
such that $\tau_{\operatorname{p}}\leq\tau$. Then the following assertions are equivalent.
\begin{enumerate}
\item $(\F,\gamma)$ has the approximation property.
\item $(\F_{\gamma}',\|\cdot\|_{\F_{\gamma}'})$ has the approximation property.
\item $\F\otimes E$ is dense in $(\FE_{\sigma},\gamma_{\sigma,\operatorname{s}}^{E})$ for every Banach space $E$ over $\K$.
\item $\F\otimes E$ is dense in $(\FE_{\sigma},\gamma_{\sigma,\operatorname{s}}^{E})$ for every complete locally 
convex Hausdorff space $E$ over $\K$.
\end{enumerate}
\end{cor}
\begin{proof}
The equivalence (a)$\Leftrightarrow$(b) follows from \prettyref{prop:approx_prop}. The remaining equivalences 
are a consequence of \prettyref{thm:linearisation_full_mixed} and \cite[Satz 10.17, p.~250]{kaballo2014}.
\end{proof}

\cite[Theorem 4.6, p.~651--652]{vargas2018} is a specical case of \prettyref{cor:approx_prop_vector_valued} 
for $\F=\operatorname{Lip}_{0}(\Omega)$. For the space $\F=H^{\infty}(\Omega)=\mathcal{H}v(\Omega)$, 
$v(z)\coloneqq 1$ for $z\in\Omega$, of bounded holomorphic $\C$-valued functions on a balanced bounded open subset $\Omega$ 
of a complex Banach space from \prettyref{cor:subspace_cont_holo_mixed} the statement of \prettyref{cor:approx_prop_vector_valued} 
is contained in \cite[5.4 Theorem, p.~883]{mujica1991}.
Further, it is known that the spaces $(H^{\infty}(\Omega),\gamma)$ and 
$(H^{\infty}(\Omega)_{\gamma}',\|\cdot\|_{H^{\infty}(\Omega)_{\gamma}'})$ 
with $\gamma=\gamma(\|\cdot\|,\tau_{\operatorname{co}})$ have the approximation property by 
\cite[Satz 3.9, p.~145]{bierstedt1973b} for simply connected open $\Omega\subset\C$ 
(cf.~\cite[V.2.4 Proposition, p.~233]{cooper1978} for $\Omega=\D\coloneqq\{z\in\C\;|\;|z|<1\}$). 
The same is true for $(\mathcal{C}v(\Omega),\gamma)$ and $(\mathcal{C}v(\Omega)_{\gamma}',\|\cdot\|_{\mathcal{C}v(\Omega)_{\gamma}'})$ by \cite[5.5 Theorem (3), (4), p.~205]{bierstedt1973a} if $\Omega$ is discrete.

We close this section with an application of \prettyref{thm:linearisation_full_mixed} to some spaces of integrable 
holomorphic functions. 
We denote by $\mathcal{H}(\D)$ the space of $\C$-valued holomorphic functions on $\D$. 
For $1\leq p<\infty$ the \emph{Hardy space} is defined by 
\[
H^{p}\coloneqq\Bigl\{f\in\mathcal{H}(\D)\;|\;
\|f\|_{p}^{p}\coloneqq \sup_{0<r<1}\frac{1}{2\pi}\int_{0}^{2\pi}|f(r\euler^{\mathrm{i}\theta})|^{p}\d\theta<\infty\Bigr\}
\]
and the \emph{weighted Bergman space} for $\alpha>-1$ by 
\[
A_{\alpha}^{p}\coloneqq\Bigl\{f\in\mathcal{H}(\D)\;|\;
\|f\|_{\alpha,p}^{p}\coloneqq \frac{\alpha+1}{\pi}\int_{\D}|f(z)|^{p}(1-|z|^2)^{\alpha}\d z<\infty\Bigr\}.
\]
The \emph{Dirichlet space} is defined by 
\[
\mathcal{D}\coloneqq\Bigl\{f\in\mathcal{H}(\D)\;|\;
\|f\|_{\mathcal{D}}^2\coloneqq |f(0)|^2+\frac{1}{\pi}\int_{\D}|f'(z)|^2\d z<\infty\Bigr\}.
\]

\begin{cor}\label{cor:hardy_bergman_dirichlet}
Let $(E,\|\cdot\|_{E})$ be a normed space over $\C$ and $(\mathcal{F}(\D),\|\cdot\|)$ be one of the spaces 
$(H^{p},\|\cdot\|_{p})$, $(A_{\alpha}^{p},\|\cdot\|_{\alpha,p})$ for $1\leq p<\infty$ and $\alpha>-1$, or 
$(\mathcal{D},\|\cdot\|_{\mathcal{D}})$. 
\begin{enumerate}
\item $(\mathcal{F}(\D,E)_{\sigma},\|\cdot\|_{\sigma}^{E},\tau_{\operatorname{co},\sigma}^{E})$ is a C-sequential Saks space 
where $\tau_{\operatorname{co},\sigma}^{E}\coloneqq (\tau_{\operatorname{co}})_{\sigma}^{E}$. 
\item If $E$ is a Banach space, then $(\mathcal{F}(\D,E)_{\sigma},\|\cdot\|_{\sigma}^{E},\tau_{\operatorname{co},\sigma}^{E})$ is complete.
\item If $E=\C$, then $(\mathcal{F}(\D),\|\cdot\|,\tau_{\operatorname{co}})$ is semi-Montel and 
$\gamma(\|\cdot\|,\tau_{\operatorname{co}})=\gamma_{\operatorname{s}}(\|\cdot\|,\tau_{\operatorname{co}})$. 
\item If $\mathcal{F}(\D)=H^{p}$, then $\tau_{\operatorname{co}}$ on $H^{p}$ is generated by the directed system of continuous 
seminorms $(|\cdot|_{p,s})_{0<s<1}$ given by 
\[
|f|_{p,s}^{p}\coloneqq \sup_{0<r<s}\frac{1}{2\pi}\int_{0}^{2\pi}|f(r\euler^{\mathrm{i}\theta})|^{p}\d\theta,\quad f\in H^{p},
\]
for $0<s<1$, which fulfils $\|f\|_{p}=\sup_{0<s<1}|f|_{p,s}$ for all $f\in H^{p}$.
\item If $\mathcal{F}(\D)=A_{\alpha}^{p}$, then $\tau_{\operatorname{co}}$ on $A_{\alpha}^{p}$ 
is generated by the directed system of continuous seminorms $(|\cdot|_{\alpha,p,r})_{0<r<1}$ given by 
\[
|f|_{\alpha,p,r}^{p}\coloneqq \frac{\alpha+1}{\pi}\int_{\D_{r}}|f(z)|^{p}(1-|z|^2)^{\alpha}\d z,\quad f\in A_{\alpha}^{p},
\]
for $0<r<1$, which fulfils $\|f\|_{\alpha,p}=\sup_{0<r<1}|f|_{\alpha,p,r}$ for all 
$f\in A_{\alpha}^{p}$, where $\D_{r}\coloneqq \{z\in\C\;|\;|z|<r\}$.
\item If $\mathcal{F}(\D)=\mathcal{D}$, then $\tau_{\operatorname{co}}$ on $\mathcal{D}$ is generated by the directed 
system of continuous seminorms $(|\cdot|_{\mathcal{D},r})_{0<r<1}$ given by 
\[
|f|_{\mathcal{D},r}^2\coloneqq |f(0)|^2+\frac{1}{\pi}\int_{\D_{r}}|f'(z)|^2\d z,
\quad f\in\mathcal{D},
\]
for $0<r<1$, which fufils $\|f\|_{\mathcal{D}}=\sup_{0<r<1}|f|_{\mathcal{D},r}$ for all $f\in\mathcal{D}$. 
\end{enumerate}
\end{cor}
\begin{proof}
(c) In all the cases we note that the $\tau_{\operatorname{co}}$-compactness of 
$B_{\|\cdot\|}$ is obtained from \cite[p.~4--5]{eklund2017} (which uses that $B_{\|\cdot\|}$ is compact in the Montel space 
$(\mathcal{H}(\D),\tau_{\operatorname{co}})$ since $B_{\|\cdot\|}$ is relatively compact there and its closedness 
is a consequence of Fatou's lemma). 
It follows that $(\mathcal{F}(\D),\|\cdot\|,\tau_{\operatorname{co}})$ is a semi-Montel Saks space 
by \prettyref{rem:pre_Saks_unit_ball_char} (a). Further, condition (ii) of \prettyref{rem:mixed=submixed} (b) 
yields that $\gamma(\|\cdot\|,\tau_{\operatorname{co}})=\gamma_{\operatorname{s}}(\|\cdot\|,\tau_{\operatorname{co}})$. 

(a)+(b) Due to part (c) and $\tau_{\operatorname{p}}\leq\tau_{\operatorname{co}}$ we have that 
$(\mathcal{F}(\D,E)_{\sigma},\|\cdot\|_{\sigma}^{E},\tau_{\operatorname{co},\sigma}^{E})$ is a Saks space, which is complete if $E$ is 
a Banach space, by \prettyref{rem:weak_vector_valued_Saks} and \prettyref{thm:linearisation_full_mixed}. 
Since the countable system of seminorms  
\[
|f|_{n}^{E}\coloneqq\sup_{e^{\ast}\in B_{\|\cdot\|_{E^{\ast}}}}\sup_{z\in\overline{\D_{1-(1/n)}}}|e^{\ast}(f(z))|,
\quad f\in \mathcal{F}(\D,E)_{\sigma},
\]
for $n\in\N$, $n\geq 2$, generates $\tau_{\operatorname{co},\sigma}^{E}$ on $\mathcal{F}(\D,E)_{\sigma}$, 
we get that $\tau_{\operatorname{co},\sigma}^{E}$ is metrisable on $\mathcal{F}(\D,E)_{\sigma}$. 
Hence  $(\mathcal{F}(\D,E)_{\sigma},\|\cdot\|_{\sigma}^{E},\tau_{\operatorname{co},\sigma}^{E})$ is C-sequential by 
\cite[Proposition 5.7, p.~2681--2682]{kruse_meichnser_seifert2018}. 

(d) For $0<s<1$ we have
\[
|f|_{p,s}^{p}=\sup_{0<r<s}\frac{1}{2\pi}\int_{0}^{2\pi}|f(r\euler^{\mathrm{i}\theta})|^{p}\d\theta
\leq \sup_{0\leq r<s}\sup_{\theta\in[0,2\pi]}|f(r\euler^{\mathrm{i}\theta})|^{p}
=\bigl(\sup_{z\in\D_{s}}|f(z)|\bigr)^{p}
\]
as well as $\|f\|_{p}=\sup_{0<s<1}|f|_{p,s}$ for all $f\in H^{p}$. 
Furthermore, for $0<s<r<1$ we remark that
\begin{align*}
 |f(z)|^{p}
&=\Bigl|f\Bigl(r\frac{z}{r}\Bigr)\Bigr|^{p}
\leq \frac{1}{2\pi}\int_{0}^{2\pi}|f(r\euler^{\mathrm{i}\theta})|^{p}\d\theta\frac{1}{1-\frac{|z|^2}{r^2}}
\leq \frac{1}{2\pi\left(1-\frac{s^2}{r^2}\right)}\int_{0}^{2\pi}|f(r\euler^{\mathrm{i}\theta})|^{p}\d\theta\\
&\leq  \frac{1}{2\pi\left(1-\frac{s^2}{r^2}\right)}|f|_{p,r}^{p}
\end{align*}
for all $z\in\D_{s}$ and $f\in H^{p}$ by the proof of \cite[Theorem 9.1, p.~253]{zhu2007}. 
It follows that $\tau_{\operatorname{co}}$ on $H^{p}$ is generated by the directed system of continuous seminorms 
$(|\cdot|_{p,s})_{0<s<1}$. 

(e) We have $\|f\|_{\alpha,p}=\sup_{0<r<1}|f|_{\alpha,p,r}$ and 
\[
    |f|_{\alpha,p,r}^{p}
\leq\frac{\alpha+1}{\pi}\max_{z\in\overline{\D}_{r}}(1-|z|^2)^{\alpha}\bigl(\sup_{z\in\D_{r}}|f(z)|\bigr)^{p}
\]
for all $f\in A_{\alpha}^{p}$ and $0<r<1$. Now, for $0<r<1$ we choose $0<s<1-r$. 
We deduce from the mean value equality for holomorphic functions and H\"older's inequality that 
\begin{align*}
 |f(z)|
&=\frac{1}{\pi s^2}\Bigl|\int_{\D_{s}(z)}f(w)\d w\Bigr|
 \leq \frac{(\pi s^2)^{\frac{1}{q}}}{\pi s^2}\Bigl(\int_{\D_{s}(z)}|f(w)|^{p}\d w\Bigr)^{\frac{1}{p}}\\
&\leq (\pi s^2)^{\frac{1}{q}-1}
     \Bigl(\int_{\D_{r+s}}|f(w)|^{p}\frac{(1-|w|^{2})^{\alpha}}{(1-|w|^{2})^{\alpha}}\d w\Bigr)^{\frac{1}{p}}\\
&\leq \frac{\pi(\pi s^2)^{\frac{1}{q}-1}}{\alpha+1}\max_{w\in\overline{\D}_{r+s}}(1-|w|^{2})^{-\frac{\alpha}{p}}
      |f|_{\alpha,p,r+s}
\end{align*}
for all $z\in\D_{r}$ and $f\in A_{\alpha}^{p}$ with $\D_{s}(z)\coloneqq \{w\in\C\;|\;|w-z|<s\}$ and 
$1<q\leq\infty$ such that $\frac{1}{p}+\frac{1}{q}=1$. 
Thus $\tau_{\operatorname{co}}$ on $A_{\alpha}^{p}$ is generated by the directed system of continuous seminorms $(|\cdot|_{\alpha,p,r})_{0<r<1}$. 

(f) We observe that $\|f\|_{\mathcal{D}}=\sup_{0<r<1}|f|_{\mathcal{D},r}$ for all $f\in\mathcal{D}$. 
Moreover, we obtain by H\"older's inequality that
\begin{align*}
     |f(z)|
&\leq |f(0)|+\bigl|\int_{0}^{z}f'(w)\d w\bigr|
 \leq |f(0)|+\int_{\D_{r}}|f'(w)|\d w\\
&\leq |f(0)|+(\pi r^2)^{\frac{1}{2}}\Bigl(\int_{\D_{r}}|f'(w)|^2\d w\Bigr)^{\frac{1}{2}}
 \leq 2^{\frac{1}{2}}\Bigl(|f(0)|^2+\pi r^2\int_{\D_{r}}|f'(w)|^2\d w\Bigr)^{\frac{1}{2}}\\
&\leq 2^{\frac{1}{2}}\pi r |f|_{\mathcal{D},r}
\end{align*}
for all $z\in\D_{r}$ and $f\in\mathcal{D}$. 
Now, for $0<r<1$ we choose $0<s<1-r$. From Cauchy's inequality we deduce the estimate 
\begin{align*}
 |f|_{\mathcal{D},r}^2
&\leq |f(0)|^2+\frac{1}{\pi}\int_{\D_{r}}|f'(z)|^2\d z
 \leq |f(0)|^2+r^2\Bigl(\sup_{z\in\D_{r}}|f'(z)|\Bigr)^2 \\
&\leq |f(0)|^2+\frac{r^2}{s^2}\Bigl(\sup_{z\in\D_{r}}\max_{w\in\C,\,|w-z|=s}|f(w)|\Bigr)^2  
 \leq \Bigl(1+\frac{r^2}{s^2}\Bigr)\Bigl(\sup_{z\in\D_{r+s}}|f(z)|\Bigr)^2
\end{align*}
for all $f\in\mathcal{D}$, implying that $\tau_{\operatorname{co}}$ on $\mathcal{D}$ is generated by the directed system of 
continuous seminorms $(|\cdot|_{\mathcal{D},r})_{0<r<1}$. 
\end{proof}

\section{Saks spaces of vector-valued functions}
\label{sect:strong_Saks_function}

This section is dedicated to Saks spaces $\FE$ of vector-valued functions which are often stronger than the spaces $\FE_{\sigma}$ 
of weak vector-valued functions from the preceding section (see \prettyref{prop:relation_weak_strong}). 
In order to derive certain systems of seminorms on such spaces which generate the mixed topology we need to recall some 
results for \emph{completely regular} Hausdorff spaces $\Omega$ (see \cite[Definition 11.1, p.~180]{james1999}). 
Examples of completely regular Hausdorff spaces are metrisable spaces by \cite[Proposition 11.5, p.~181]{james1999} and 
locally convex Hausdorff spaces by \cite[Proposition 3.27, p.~95]{fabian2011}. 
Further, every subspace of a completely regular Hausdorff space is completely regular and Hausdorff as well. 
For a completely regular Hausdorff space $\Omega$ we denote by $\mathcal{W}_{\operatorname{b},0}^{+}(\Omega)$ the family 
of all bounded functions $w\colon \Omega \to [0,\infty)$ that \emph{vanish at infinity}, 
i.e.~for every $\varepsilon >0$ the set $\{x\in\Omega\;|\;w(x)\geq \varepsilon\}$ is compact. 
Further, we denote by $\mathcal{W}_{\operatorname{usc},0}^{+}(\Omega)$ resp.~$\mathcal{C}_{0}^{+}(\Omega)$ 
the family of all upper semicontinuous resp.~continuous functions $w\colon \Omega \to [0,\infty)$ that vanish at infinity. 
We note that $\mathcal{C}_{0}^{+}(\Omega)\subset\mathcal{W}_{\operatorname{usc},0}^{+}(\Omega)
\subset\mathcal{W}_{\operatorname{b},0}^{+}(\Omega)$ 
because upper semicontinuous functions are bounded on compact sets. 
By the proofs of \cite[II.1.11 Proposition, p.~82--83]{cooper1978} and \cite[Proposition 3, p.~590]{cooper1971} 
we have the following proposition. 

\begin{prop}\label{prop:weight_strict_topology}
Let $\Omega$ be a completely regular Hausdorff space, $(K_{n})_{n\in\N}$ a strictly increasing sequence of compact subsets 
of $\Omega$ and $(a_{n})_{n\in\N}$ a strictly decreasing positive null-sequence. 
Then there is $w\in\mathcal{W}_{\operatorname{usc},0}^{+}(\Omega)$ such that 
$\operatorname{supp} w\subset \bigcup_{n\in\N}K_{n}$ and
$w(x)=a_{1}$ for $x\in K_{1}$ and $a_{n+1}\leq w(x)\leq a_{n}$ for $x\in K_{n+1}\setminus K_{n}$ and $n\in\N$. 
If $\Omega$ is locally compact and $K_{n}\subset\mathring{K}_{n+1}$ for every $n\in\N$, then we may choose 
$w\in\mathcal{C}_{0}^{+}(\Omega)$. 
\end{prop}

Here, $\operatorname{supp} w$ denotes the support of $w$ and $\mathring{K}_{n+1}$ the set of inner points of $K_{n+1}$. 

\begin{defn}\label{defn:standard_space}
Let $\Omega$ and $\Lambda$ be non-empty sets, $v\colon\Lambda\to (0,\infty)$, $(E,\|\cdot\|_{E})$ a normed space 
over $\K$, $\GE$ a linear subspace of $E^{\Omega}$, 
$q^{E}\colon\GE\to[0,\infty)$ a seminorm and $T^{E}\colon \GE\to E^{\Lambda}$ a linear map 
where $E^{\Lambda}$ denotes the space of functions from $\Lambda$ to $E$. 
We define the space
\[
\FVE\coloneqq\{f\in \GE\;|\;\|f\|^{E}\coloneqq q^{E}(f)+\sup_{x \in \Lambda}\|T^{E}(f)(x)\|_{E}v(x)<\infty\}.
\]
If $E=\K$, we write $\G\coloneqq \mathcal{G}(\Omega,\K)$, $\FV\coloneqq \mathcal{F}v(\Omega,\K)$, $q\coloneqq q^{\K}$ and 
$T\coloneqq T^{\K}$ and $\|\cdot\|\coloneqq\|\cdot\|^{\K}$. 
If we want to emphasize dependencies, we write $\|f\|_{\FVE}$ instead of $\|f\|^{E}$.
\end{defn}

For a non-empty set $\Lambda$ we denote by $\mathcal{N}_{\Lambda}$ the familiy of finite subsets of $\Lambda$. 
If $\Lambda$ is a topological space, we denote by $\mathcal{K}_{\Lambda}$ the family of compact subsets of $\Lambda$. 

\begin{thm}\label{thm:standard_saks_space}
Let $\Omega$ and $\Lambda$ be non-empty sets, $v\colon\Lambda\to (0,\infty)$, $(E,\|\cdot\|_{E})$ a normed space over $\K$, 
$\GE$ a linear subspace of $E^{\Omega}$, $q^{E}\colon\GE\to[0,\infty)$ a seminorm, $T^{E}\colon \GE\to E^{\Lambda}$ a linear map 
and suppose that $(\FVE,\|\cdot\|^{E})$ is normed. 
Let $\mathcal{S}$ be a family of subsets of $\Lambda$ such that $\mathcal{S}$ is closed under finite unions, 
$\Lambda=\bigcup_{S\in\mathcal{S}}S$ and denote by 
$\tau_{\mathcal{S}}^{E}$ the locally convex Hausdorff topology generated by the directed system of seminorms
\[
q_{S}^{E}(f)\coloneqq q^{E}(f)+\sup_{x\in S}\|T^{E}(f)(x)\|_{E}v(x),\quad f\in\FVE,
\]
for $S\in\mathcal{S}$. Then the following assertions hold.
\begin{enumerate}
\item $(\FVE,\|\cdot\|^{E},\tau_{\mathcal{S}}^{E})$ is a Saks space. 
\item If there is a sequence $(S_{n})_{n\in\N}$ in $\mathcal{S}$ such that for every $S\in\mathcal{S}$ there is $N\in\N$ 
with $S\subset S_{N}$, then $(\FVE,\|\cdot\|^{E},\tau_{\mathcal{S}}^{E})$ is C-sequential. 
\item The submixed topology $\gamma_{\operatorname{s}}(\|\cdot\|^{E},\tau_{\mathcal{S}}^{E})$ is generated by the 
system of seminorms
\[
\vertiii{f}_{(S_{n},a_{n})_{n\in\N}}^{E}\coloneqq\sup_{n\in\N}\sup_{x\in S_{n}}(q^{E}(f)+\|T^{E}(f)(x)\|_{E}v(x))a_{n},\quad f\in \FVE,
\]
where $(S_{n})_{n\in\N}$ is a sequence in $\mathcal{S}$ and $(a_{n})_{n\in\N}\in c_{0}^{+}$.
\item If $\mathcal{S}=\mathcal{N}_{\Lambda}$, then $\gamma_{\operatorname{s}}(\|\cdot\|^{E},\tau_{\mathcal{N}_{\Lambda}}^{E})$ is generated by the system of seminorms
\[
\vertiii{f}_{(x_{n},a_{n})_{n\in\N}}^{E}\coloneqq\sup_{n\in\N}(q^{E}(f)+\|T^{E}(f)(x_{n})\|_{E}v(x_{n}))a_{n},\quad f\in \FVE,
\]
where $(x_{n})_{n\in\N}$ is a sequence in $\Lambda$ and $(a_{n})_{n\in\N}\in c_{0}^{+}$.
\item Let $\Lambda$ be a completely regular Hausdorff space and set
$\mathcal{W}_{0}\coloneqq\mathcal{W}_{\operatorname{b},0}^{+}(\Lambda)$ or $\mathcal{W}_{\operatorname{usc},0}^{+}(\Lambda)$. 
If $\mathcal{S}=\mathcal{K}_{\Lambda}$ and $\gamma(\|\cdot\|^{E},\tau_{\mathcal{K}_{\Lambda}}^{E})=\gamma_{\operatorname{s}}(\|\cdot\|^{E},\tau_{\mathcal{K}_{\Lambda}}^{E})$, then the mixed topology 
$\gamma(\|\cdot\|^{E},\tau_{\mathcal{K}_{\Lambda}}^{E})$ is generated by the system of seminorms
\[
|f|_{w}^{E}\coloneqq\sup_{x\in\Lambda}(q^{E}(f)+\|T^{E}(f)(x)\|_{E}v(x))w(x),\quad f \in\FVE,
\]
for $w\in\mathcal{W}_{0}$. If $\Lambda$ is locally compact, we may replace $\mathcal{W}_{0}$ by $\mathcal{C}_{0}^{+}(\Lambda)$. 
\end{enumerate}
\end{thm}
\begin{proof}
(a) First, we note that the system of seminorms $(q_{S}^{E})_{S\in\mathcal{S}}$ is directed and Hausdorff 
since $\mathcal{S}$ is closed under finite unions, $\|\cdot\|^{E}$ a norm by assumption and
\[
 \|f\|^{E}=q^{E}(f)+\sup_{S\in\mathcal{S}}q_{S}^{E}(f)
\]
for all $f\in\FVE$ because $\Lambda=\bigcup_{S\in\mathcal{S}}S$. 
Hence $\tau_{\mathcal{S}}^{E}$ is a locally convex Hausdorff topology with 
$\tau_{\mathcal{S}}^{E}\leq \tau_{\|\cdot\|^{E}}$ and $(\FVE,\|\cdot\|^{E},\tau_{\mathcal{S}}^{E})$ a Saks space.

(b) The countable system of seminorms $(q_{S_{n}})_{n\in\N}$ generates $\tau_{\mathcal{S}}^{E}$ by our assumption on $\mathcal{S}$. 
Hence $\tau_{\mathcal{S}}^{E}$ is metrisable and so $(\FVE,\gamma(\|\cdot\|^{E},\tau_{\mathcal{S}}^{E}))$ C-sequential by 
\cite[Proposition 5.7, p.~2681--2682]{kruse_meichnser_seifert2018}.

(c) This follows from part (a) and the definition of $\gamma_{\operatorname{s}}(\|\cdot\|^{E},\tau_{\mathcal{S}}^{E})$.

(d) Let $(a_{n})_{n\in\N}\in c_{0}^{+}$ and $(S_{n})_{n\in\N}$ be a sequence of finite subsets of $\Lambda$ with cardinality 
$m_{n}\coloneqq|S_{n}|$ for $n\in\N$. 
Then every $S_{n}$, $n\in\N$, is of the form $S_{n}=\{s^{n}_{1},\ldots,s^{n}_{m_{n}}\}$ with distinct elements 
$s^{n}_{i}\in\Lambda$ for $1\leq i\leq m_{n}$. We set $x_{i}\coloneqq s^{1}_{i}$ for $1\leq i\leq m_{1}$. Further, for $n\in\N$ 
we set $j_{n}\coloneqq \sum_{l=1}^{n}m_{l}$ and $x_{j_{n}+i}\coloneqq s^{n+1}_{i}$ for $1\leq i\leq m_{n+1}$. 
Moreover, we set $b_{i}\coloneqq a_{1}$ for $1\leq i\leq m_{1}$. For $n\in\N$ 
we set $b_{j_{n}+i}\coloneqq a_{n+1}$ for $1\leq i\leq m_{n+1}$. Then we have $(b_{n})_{n\in\N}\in c_{0}^{+}$ and  
$\vertiii{f}_{(S_{n},a_{n})_{n\in\N}}^{E}=\vertiii{f}_{(x_{n},b_{n})_{n\in\N}}^{E}$ for all $f\in\FVE$. 
On the other hand, $\vertiii{f}_{(\{z_{n}\},a_{n})_{n\in\N}}^{E}=\vertiii{f}_{(z_{n},a_{n})_{n\in\N}}^{E}$ for all $f\in\FVE$ 
and every sequence $(z_{n})_{n\in\N}$ in $\Lambda$. Thus statement (d) follows from part (c). 

(e) We denote by $\omega_{\operatorname{b}}^{E}$ and $\omega_{\operatorname{usc}}^{E}$ the locally convex Hausdorff topologies 
generated by $(|\cdot|_{w}^{E})_{w\in\mathcal{W}_{\operatorname{b},0}^{+}(\Lambda)}$ 
and $(|\cdot|_{w}^{E})_{w\in\mathcal{W}_{\operatorname{usc},0}^{+}(\Lambda)}$, respectively. 
First, we prove that the identity map $\id\colon(\FVE,\gamma(\|\cdot\|^{E},\tau_{\mathcal{K}_{\Lambda}}^{E}))\to 
(\FVE,\omega_{\operatorname{b}}^{E})$ is continuous. 
Due to \cite[I.1.7 Corollary, p.~8]{cooper1978} and \cite[I.1.8 Lemma, p.~8]{cooper1978} 
we only need to prove that its restriction to $B_{\|\cdot\|^{E}}$ is $\tau_{\mathcal{K}_{\Lambda}}^{E}$-continuous at zero. 
Let $\varepsilon >0$, $w\in\mathcal{W}_{\operatorname{b},0}^{+}(\Lambda)$ and set 
$V\coloneqq\{f\in\FVE\;|\;|f|_{w}^{E}\leq\varepsilon\}$. 
Then there is a compact set $K\subset\Lambda$ such that $w(x)<\frac{\varepsilon}{2}$ for $x\in\Lambda\setminus K$. 
We define $U\coloneqq \{f\in\FVE\;|\;q_{K}^{E}(f)\leq\frac{\varepsilon}{2(1+\|w\|_{\infty})}\}$ 
where $\|w\|_{\infty}\coloneqq\sup_{x\in\Lambda}w(x)$ and note that for all $f\in U\cap B_{\|\cdot\|^{E}}$ it holds that
\begin{align*}
     |f|_{w}^{E}
&\leq \sup_{x\in\Lambda\setminus K}(q^{E}(f)+\|T^{E}(f)(x)\|_{E}v(x))w(x)\\
&\phantom{\leq} +\sup_{x\in K}(q^{E}(f)+\|T^{E}(f)(x)\|_{E}v(x))w(x)\\
&\leq \frac{\varepsilon}{2}\|f\|^{E}+\|w\|_{\infty}q_{K}^{E}(f)
\leq \frac{\varepsilon}{2}+\|w\|_{\infty}\frac{\varepsilon}{2(1+\|w\|_{\infty})}\leq \varepsilon,
\end{align*}
yielding $(U\cap B_{\|\cdot\|^{E}})\subset V$ and so the continuity of $\id$. 

Second, we prove that $\id\colon(\FVE,\omega_{\operatorname{usc}}^{E})\to(\FVE,\gamma(\|\cdot\|^{E},\tau_{\mathcal{K}_{\Lambda}}^{E}))$ is continuous. 
Let $(K_{n})_{n\in\N}$ be a sequence of compact subsets of $\Lambda$ and $(a_{n})_{n\in\N}\in c_{0}^{+}$. 
W.l.o.g.~$K_{n}\subset K_{n+1}$ and $0<a_{n+1}<a_{n}$ for all $n\in\N$.  
Then there is $w\in\mathcal{W}_{\operatorname{usc},0}^{+}(\Lambda)$ with $\operatorname{supp} w\subset \bigcup_{n\in\N}K_{n}$ 
such that $w(x)=a_{1}$ for $x\in K_{1}$ and $a_{n+1}\leq w(x)\leq a_{n}$ for $x\in K_{n+1}\setminus K_{n}$ 
by \prettyref{prop:weight_strict_topology}. 
It follows that
\[
\vertiii{f}_{(K_{n},a_{n})_{n\in\N}}^{E}
\leq\sup_{x\in\Lambda}(q^{E}(f)+\|T^{E}(f)(x)\|_{E}v(x))w(x)=|f|_{w}^{E}
\]
for all $f\in\FVE$, which yields the continuity of $\id$ by part (c) and the assumption 
$\gamma(\|\cdot\|^{E},\tau_{\mathcal{K}_{\Lambda}}^{E})=\gamma_{\operatorname{s}}(\|\cdot\|^{E},\tau_{\mathcal{K}_{\Lambda}}^{E})$. 
Due to $\mathcal{C}_{0}^{+}(\Lambda)\subset\mathcal{W}_{\operatorname{usc},0}^{+}(\Lambda)
\subset\mathcal{W}_{\operatorname{b},0}^{+}(\Lambda)$ and \prettyref{prop:weight_strict_topology} this proves statement (e).
\end{proof}

The proof of \prettyref{thm:standard_saks_space} (e) is just an adaptation of the proofs of 
\cite[Proposition 3, p.~590]{cooper1971} and \cite[II.1.11 Proposition, p.~82]{cooper1978}. 
In the next proposition we show that the space $\FVE$ from \prettyref{defn:standard_space} is a linear subspace of 
$\FVE_{\sigma}$ under some mild assumptions and we use the topology $\tau_{\mathcal{S}}^{E}$ 
defined in \prettyref{thm:standard_saks_space}.

\begin{prop}\label{prop:relation_weak_strong}
Let $\Omega$ and $\Lambda$ be non-empty sets, $v\colon\Lambda\to (0,\infty)$, $(E,\|\cdot\|_{E})$ a normed space over $\K$, 
$\G$ a linear subspace of $\K^{\Omega}$, $\GE$ a linear subspace of $E^{\Omega}$, $q\colon\G\to[0,\infty)$ 
and $q^{E}\colon\GE\to[0,\infty)$ seminorms, $T\colon \G\to\K^{\Lambda}$ and $T^{E}\colon \GE\to E^{\Lambda}$ linear maps 
and $(\FV,\|\cdot\|)$ a normed space such that 
\begin{enumerate}
\item[(i)] $e^{\ast}\circ f\in\FV$ and $(e^{\ast}\circ T^{E})(f)=T(e^{\ast}\circ f)$ for all $e^{\ast}\in E^{\ast}$ and $f\in\FVE$,
\item[(ii)] $q^{E}(f)=\sup_{e^{\ast}\in B_{\|\cdot\|_{E^{\ast}}}}q(e^{\ast}\circ f)$ for all $f\in\FVE$.
\end{enumerate}
Then the following assertions hold.
\begin{enumerate}
\item $\FVE$ is a linear subspace of $\FVE_{\sigma}$ and $\|\cdot\|_{\sigma}^{E}\leq\|\cdot\|^{E}\leq 2\|\cdot\|_{\sigma}^{E}$ 
on $\FVE$. In particular, $\|\cdot\|^{E}$ is a norm on $\FVE$. If $q=0$, then $\|\cdot\|_{\sigma}^{E}=\|\cdot\|^{E}$ on $\FVE$.
\end{enumerate}
Suppose for (b)--(d) that $\mathcal{S}$ is a family of subsets of $\Lambda$ such that $\mathcal{S}$ is closed under finite unions and 
$\Lambda=\bigcup_{S\in\mathcal{S}}S$, and $(\FV,\|\cdot\|,\tau_{\mathcal{S}})$ is semi-Montel 
where $\tau_{\mathcal{S}}\coloneqq\tau_{\mathcal{S}}^{\K}$.
\begin{enumerate}
\item[(b)] Then $\gamma_{\operatorname{s}}(\|\cdot\|^{E},\tau_{\mathcal{S}}^{E})
=\gamma_{\operatorname{s}}(\|\cdot\|_{\sigma}^{E},\tau_{\mathcal{S},\sigma}^{E})$ on $\FVE$ where 
$\tau_{\mathcal{S},\sigma}^{E}\coloneqq (\tau_{\mathcal{S}})_{\sigma}^{E}$. 
\item[(c)] If $\tau_{\operatorname{p}}\leq\tau_{\mathcal{S}}$, $\FVE=\FVE_{\sigma}$ (as linear spaces) and $E$ is a Banach space, then $(\FVE,\gamma_{\operatorname{s}}(\|\cdot\|^{E},\tau_{\mathcal{S}}^{E}))$ 
and $(\FVE,\|\cdot\|^{E},\tau_{\mathcal{S}}^{E})$ are complete.
\item[(d)] If $\Lambda$ is a completely regular Hausdorff space, $\mathcal{S}=\mathcal{K}_{\Lambda}$ and 
$\mathcal{W}_{0}\coloneqq\mathcal{W}_{\operatorname{b},0}^{+}(\Lambda)$ or $\mathcal{W}_{\operatorname{usc},0}^{+}(\Lambda)$, 
then $\gamma_{\operatorname{s}}(\|\cdot\|^{E},\tau_{\mathcal{K}_{\Lambda}}^{E})$ is generated by 
the system of seminorms $(|\cdot|_{w}^{E})_{w\in\mathcal{W}_{0}}$. 
If $\Lambda$ is locally compact, we may replace $\mathcal{W}_{0}$ by $\mathcal{C}_{0}^{+}(\Lambda)$. 
\end{enumerate}
\end{prop}
\begin{proof}
(a) Due to the first part of condition (i) we obtain that $\FVE$ is a linear subspace of $\FVE_{\sigma}$. 
The second part of condition (i) implies that 
\begin{align*}
 \sup_{x\in\Lambda}\|T^{E}(f)(x)\|_{E}v(x)
&=\sup_{x\in\Lambda}\sup_{e^{\ast}\in B_{\|\cdot\|_{E^{\ast}}}}|(e^{\ast}\circ T^{E})(f)(x)|v(x)\\
&=\sup_{e^{\ast}\in B_{\|\cdot\|_{E^{\ast}}}}\sup_{x\in\Lambda}|T(e^{\ast}\circ f)(x)|v(x)
\end{align*}
for all $f\in\FVE$. Together with condition (ii) this yields that 
$\|\cdot\|_{\sigma}^{E}\leq\|\cdot\|^{E}\leq 2\|\cdot\|_{\sigma}^{E}$ on $\FVE$, and 
$\|\cdot\|_{\sigma}^{E}=\|\cdot\|^{E}$ on $\FVE$ if additionally $q=0$. Since $\|\cdot\|_{\sigma}^{E}$ is a norm on 
$\FVE_{\sigma}$, we get that $\|\cdot\|^{E}$ is a norm on $\FVE$. 

(b) Let $(S_{n})_{n\in\N}$ be a sequence in $\mathcal{S}$ and $(a_{n})_{n\in\N}\in c_{0}^{+}$. 
We have by part (a)
\[
 \vertiii{f}_{\sigma,(S_{n},a_{n})_{n\in\N},\|\cdot\|_{E}}
 \underset{\eqref{eq:weak_submixed}}{=}\sup_{e^{\ast}\in B_{\|\cdot\|_{E^{\ast}}}}\vertiii{e^{\ast}\circ f}_{(S_{n},a_{n})_{n\in\N}}^{\K}
 =\sup_{n\in\N}\sup_{e^{\ast}\in B_{\|\cdot\|_{E^{\ast}}}}q_{S_{n}}^{\K}(e^{\ast}\circ f)a_{n}
\]
for all $f\in\FVE$. Using the second part of condition (i) we get as in part (a) that
\begin{align*}
 \sup_{x\in S_{n}}\|T^{E}(f)(x)\|_{E}v(x)=\sup_{e^{\ast}\in B_{\|\cdot\|_{E^{\ast}}}}\sup_{x\in S_{n}}|T(e^{\ast}\circ f)(x)|v(x)
\end{align*}
for all $f\in\FVE$. In combination with condition (ii) we obtain the estimates 
$\vertiii{\cdot}_{\sigma,(S_{n},a_{n})_{n\in\N},\|\cdot\|_{E}}\leq \vertiii{\cdot}_{(S_{n},a_{n})_{n\in\N}}^{E}\leq 
2\vertiii{\cdot}_{\sigma,(S_{n},a_{n})_{n\in\N},\|\cdot\|_{E}}$ on $\FVE$, which imply
$\gamma_{\operatorname{s}}(\|\cdot\|^{E},\tau_{\mathcal{S}}^{E})
=\gamma_{\operatorname{s}}(\|\cdot\|_{\sigma}^{E},\tau_{\mathcal{S},\sigma}^{E})$ on $\FVE$ by 
\prettyref{rem:weak_vector_valued_Saks} and \prettyref{thm:standard_saks_space} (b).

(c) The completeness of $(\FVE,\gamma_{\operatorname{s}}(\|\cdot\|^{E},\tau_{\mathcal{S}}^{E}))$ follows from part (b), 
\prettyref{rem:weak_vector_valued_Saks} and \prettyref{thm:linearisation_full_mixed}. The completeness of 
$(\FVE,\|\cdot\|^{E},\tau_{\mathcal{S}}^{E})$ then follows from $\gamma_{\operatorname{s}}(\|\cdot\|^{E},\tau_{\mathcal{S}}^{E})\leq\gamma(\|\cdot\|^{E},\tau_{\mathcal{S}}^{E})$ by \prettyref{rem:mixed=submixed} (a). 

(d) As $(\FV,\|\cdot\|,\tau_{\mathcal{K}_{\Lambda}})$ is semi-Montel, we know that 
$\gamma(\|\cdot\|,\tau_{\mathcal{K}_{\Lambda}})=\gamma_{\operatorname{s}}(\|\cdot\|,\tau_{\mathcal{K}_{\Lambda}})$ 
by \prettyref{rem:pre_Saks_unit_ball_char} (a) and condition (ii) of \prettyref{rem:mixed=submixed} (b). 
Due to \prettyref{thm:standard_saks_space} (d) $\gamma(\|\cdot\|,\tau_{\mathcal{K}_{\Lambda}})$ 
is generated by the system of seminorms
\[
|f|_{w}^{\K}=\sup_{x\in\Lambda}(q(f)+|T(f)(x)|v(x))w(x),\quad f \in\FV,
\]
for $w\in\mathcal{W}_{0}$. We deduce that the system of seminorms given by
\[
|f|_{\sigma,w,\|\cdot\|_{E}}\coloneqq\sup_{e^{\ast}\in B_{\|\cdot\|_{E^{\ast}}}}\sup_{x\in\Lambda}(q(e^{\ast}\circ f)+|T(e^{\ast}\circ f)(x)|v(x))w(x),\quad f \in\FVE_{\sigma},
\]
for $w\in\mathcal{W}_{0}$ generates the topology $\gamma_{\operatorname{s}}(\|\cdot\|_{\sigma}^{E},\tau_{\mathcal{K}_{\Lambda},\sigma}^{E})$ 
on $\FVE_{\sigma}$ by \prettyref{rem:weak_vector_valued_ind_of_choice_seminorms} and \prettyref{rem:weak_vector_valued_Saks}. 
Similar to the proofs of part (a) and (b) we obtain that $|\cdot|_{\sigma,w,\|\cdot\|_{E}}\leq |\cdot|_{w}^{E}\leq 
2|\cdot|_{\sigma,w,\|\cdot\|_{E}}$ on $\FVE$ for every $w\in\mathcal{W}_{0}$, 
yielding that $\gamma_{\operatorname{s}}(\|\cdot\|^{E},\tau_{\mathcal{K}_{\Lambda}}^{E})$ is generated by 
the system of seminorms $(|\cdot|_{w}^{E})_{w\in\mathcal{W}_{0}}$ by part (b).
\end{proof}

Condition (i) of \prettyref{prop:relation_weak_strong} means that the tuple $(T^{E},T)$ is strong for $(\mathcal{F}v,E)$ in the sense 
of \cite[Definition 2.2 (b), p.~4]{kruse2019_3}.
In our first example of this section we consider weighted Saks spaces of continuous vector-valued functions 
on a completely regular Hausdorff space and a sufficient condition for their completeness involves the notion of a $k_{\R}$-space. 
A completely regular space $\Omega$ is called a \emph{$k_{\R}$-space} if for any 
completely regular space $Y$ and any map $f\colon \Omega \to Y$, 
whose restriction to each compact $K\subset\Omega$ is continuous, the map is already continuous on $\Omega$ 
(see \cite[(2.3.7) Proposition, p.\ 22]{buchwalter1969}). Moreover, a topological space $\Omega$ is called a \emph{$k$-space} 
if it fulfils the following condition: $A\subset \Omega$ is closed if and only if 
$A\cap K$ is closed in $K$ for every compact $K\subset\Omega$.  
Examples of Hausdorff $k_{\R}$-spaces are completely regular Hausdorff $k$-spaces by \cite[3.3.21 Theorem, p.\ 152]{engelking1989}. 
In particular, metrisable spaces and \emph{DFM-spaces}, i.e.~strong duals of Fr\'echet--Montel spaces, 
are completely regular Hausdorff $k$-spaces by \cite[3.3.20 Theorem, p.~152]{engelking1989} 
and \cite[4.11 Theorem (5), p.~39]{kriegl1997}, respectively. 
For a non-empty completely regular Hausdorff space $\Omega$, a continuous function $v\colon\Omega\to (0,\infty)$ and a 
normed space $(E,\|\cdot\|_{E})$ over $\K$ we set 
\[
\mathcal{C}v(\Omega,E)\coloneqq \{f\in\mathcal{C}(\Omega,E)\;|\;\|f\|^{E}\coloneqq\sup_{x\in\Omega}\|f(x)\|_{E}v(x)<\infty\}
\]
where $\mathcal{C}(\Omega,E)$ is the space $E$-valued continuous functions on $\Omega$. 
Further, we define $\mathcal{C}v(\Omega)\coloneqq\mathcal{C}v(\Omega,\K)$ and $\|\cdot\|\coloneqq \|\cdot\|^{\K}$.
Setting $\Lambda\coloneqq\Omega$, $\mathcal{G}(\Omega,E)\coloneqq\mathcal{C}(\Omega,E)$, $q^{E}\coloneqq 0$ and 
$T^{E}(f)\coloneqq f$ for $f\in\mathcal{C}(\Omega,E)$, we note that $\FVE=\mathcal{C}v(\Omega,E)$,
$\tau_{\mathcal{K}_{\Omega}}^{E}=\tau_{\operatorname{co}}^{E}$ on $\mathcal{C}v(\Omega,E)$ and conditions (i) and (ii) of 
\prettyref{prop:relation_weak_strong} are fulfilled. 

\begin{cor}\label{cor:cv_mixed}
Let $\Omega$ be a non-empty completely regular Hausdorff space, $v\colon\Omega\to (0,\infty)$ continuous 
and $(E,\|\cdot\|_{E})$ a normed space over $\K$. Then the following assertions hold. 
\begin{enumerate}
\item $(\mathcal{C}v(\Omega,E),\|\cdot\|^{E},\tau_{\operatorname{co}}^{E})$ is a Saks space, $\gamma(\|\cdot\|^{E},\tau_{\operatorname{co}}^{E})=\gamma_{\operatorname{s}}(\|\cdot\|^{E},\tau_{\operatorname{co}}^{E})$ and the mixed topology
$\gamma(\|\cdot\|^{E},\tau_{\operatorname{co}}^{E})$ is generated by the system of seminorms
\[
\vertiii{f}_{(K_{n},a_{n})_{n\in\N}}^{E}\coloneqq\sup_{n\in\N}\sup_{x\in K_{n}}\|f(x)\|_{E}v(x)a_{n},\quad f\in\mathcal{C}v(\Omega,E),
\]
where $(K_{n})_{n\in\N}$ is a sequence of compact subsets of $\Omega$ and $(a_{n})_{n\in\N}\in c_{0}^{+}$.
\item If $\Omega$ is discrete, then $\mathcal{C}v(\Omega,E)=\mathcal{C}v(\Omega,E)_{\sigma}$, 
$\gamma(\|\cdot\|^{E},\tau_{\operatorname{co}}^{E})=\gamma_{\operatorname{s}}(\|\cdot\|^{E},\tau_{\mathcal{N}_{\Omega}}^{E})$ and 
$\gamma(\|\cdot\|^{E},\tau_{\operatorname{co}}^{E})$ is also generated by the system of seminorms
\[
\vertiii{f}_{(x_{n},a_{n})_{n\in\N}}^{E}\coloneqq\sup_{n\in\N}\|f(x_{n})\|_{E}v(x_{n})a_{n},\quad f\in\mathcal{C}v(\Omega,E),
\]
where $(x_{n})_{n\in\N}$ is a sequence in $\Omega$ and $(a_{n})_{n\in\N}\in c_{0}^{+}$.
\item If $\Omega$ is a $k_{\R}$-space and $E$ a Banach space, then 
$(\mathcal{C}v(\Omega,E),\|\cdot\|^{E},\tau_{\operatorname{co}}^{E})$ is complete.
\item If $\Omega$ is \emph{hemicompact}, i.e.~there is a sequence $(K_{n})_{n\in\N}$ in $\mathcal{K}_{\Omega}$ such that for every $K\in\mathcal{K}_{\Omega}$ there is $N\in\N$ with $K\subset K_{N}$, then $(\mathcal{C}v(\Omega,E),\|\cdot\|^{E},\tau_{\operatorname{co}}^{E})$ is C-sequential.
\item If $\Omega$ is a hemicompact $k_{\R}$-space, or a completely metrisable space, then $(\mathcal{C}v(\Omega,E),\gamma(\|\cdot\|^{E},\tau_{\operatorname{co}}^{E}))$ is a Mackey space. If $E$ is in addition a Banach space, then 
$(\mathcal{C}v(\Omega,E),\gamma(\|\cdot\|^{E},\tau_{\operatorname{co}}^{E}))$ is a strong Mackey space.
\item Let $\mathcal{W}_{0}\coloneqq\mathcal{W}_{\operatorname{b},0}^{+}(\Omega)$ or $\mathcal{W}_{\operatorname{usc},0}^{+}(\Omega)$. 
Then $\gamma(\|\cdot\|^{E},\tau_{\operatorname{co}}^{E})$ is also generated by the system of seminorms
\[
|f|_{w}^{E}\coloneqq\sup_{x\in\Omega}\|f(x)\|_{E}v(x)w(x),\quad f \in\mathcal{C}v(\Omega,E),
\]
for $w\in\mathcal{W}_{0}$. If $\Omega$ is locally compact, we may replace $\mathcal{W}_{0}$ by $\mathcal{C}_{0}^{+}(\Omega)$. 
\end{enumerate}
\end{cor}
\begin{proof}
Due \prettyref{thm:standard_saks_space} and our observations above we only need to prove that 
$\gamma(\|\cdot\|^{E},\tau_{\operatorname{co}}^{E})=\gamma_{\operatorname{s}}(\|\cdot\|^{E},\tau_{\operatorname{co}}^{E})$ 
from part (a), and parts (b), (c) and (e). 

(a) We show that condition (i) of \prettyref{rem:mixed=submixed} (b) is fulfilled. We only need to adjust the proof of 
\cite[Example D), p.~65--66]{wiweger1961} to the weighted vector-valued case, which we do for the sake of the reader. 
Let $f\in\mathcal{C}v(\Omega,E)$, $\varepsilon>0$ and $K\subset\Omega$ be compact. Since $\|f(\cdot)\|_{E}v$ is continuous 
on $\Omega$, there is an open set $G\subset\Omega$ with $K\subset G$ such that 
\begin{equation}\label{eq:cv_decomp}
\sup_{x\in G}\|f(x)\|_{E}v(x)\leq \sup_{x\in K}\|f(x)\|_{E}v(x)+\varepsilon =q_{K}^{E}(f)+\varepsilon.
\end{equation}
The set $\Omega\setminus G$ is closed and disjoint with the compact set $K\subset\Omega$. 
By \cite[(2.1.5) Proposition, p.~17]{buchwalter1969} the complete regularity of $\Omega$ implies 
that there is a continuous function $u\colon\Omega\to [0,1]$ such that $u_{\mid K}=0$ and $u_{\mid\Omega\setminus G}=1$. 
Now, we set $g\coloneqq (1-u)f$ and $h\coloneqq uf$ and note that $f=g+h$ and $g,h\in\mathcal{C}v(\Omega,E)$. 
Due to the properties of $u$ we have $q_{K}^{E}(h)=0$ and 
\begin{align*}
 \|g\|^{E}
&=\sup_{x\in \Omega}(1-u(x))\|f(x)\|_{E}v(x)
 =\sup_{x\in G}(1-u(x))\|f(x)\|_{E}v(x)
 \leq \sup_{x\in G}\|f(x)\|_{E}v(x)\\
&\underset{\mathclap{\eqref{eq:cv_decomp}}}{\leq} q_{K}^{E}(f)+\varepsilon .
\end{align*}
Thus condition (i) of \prettyref{rem:mixed=submixed} (b) is fulfilled, yielding 
$\gamma(\|\cdot\|^{E},\tau_{\operatorname{co}}^{E})=\gamma_{\operatorname{s}}(\|\cdot\|^{E},\tau_{\operatorname{co}}^{E})$.

(b) Since $\Omega$ is discrete, every subset of $\Omega$ is open, and a subset of $\Omega$ is compact if and only if it is finite. 
Thus $\tau_{\operatorname{co}}^{E}=\tau_{\mathcal{K}_{\Omega}}^{E}=\tau_{\mathcal{N}_{\Omega}}^{E}$ and statement (b) follows 
from part (a), \prettyref{thm:standard_saks_space} (d) and Mackey's theorem. 

(c) Let $\mathcal{C}_{b}(\Omega,E)\coloneqq \mathcal{C}\widetilde{v}(\Omega,E)$ for $\widetilde{v}(x)\coloneqq 1$, $x\in\Omega$, 
and set $\|\cdot\|_{\infty}^{E}\coloneqq\|\cdot\|_{\mathcal{C}\widetilde{v}(\Omega,E)}$. 
By part (f) the multiplication operator 
\[
M_{v}^{E}\colon \mathcal{C}v(\Omega,E)\to \mathcal{C}_{b}(\Omega,E),\;M_{v}^{E}(f)\coloneqq fv,
\]
is a topological isomorphism w.r.t.~$\gamma(\|\cdot\|^{E},\tau_{\operatorname{co}}^{E})$ and 
$\gamma(\|\cdot\|_{\infty}^{E},\tau_{\operatorname{co}}^{E})$. 
The space $(\mathcal{C}_{b}(\Omega,E),\gamma(\|\cdot\|_{\infty}^{E},\tau_{\operatorname{co}}^{E}))$ is complete by 
\cite[II.4.2 Proposition 2), p.~113]{cooper1978} because $\Omega$ is a $k_{\R}$-space (see \cite[p.~112]{cooper1978} and note that 
$k_{\R}$-spaces are exactly the $\mathcal{K}$-complete spaces by \cite[p.~80]{cooper1978}) and 
$(E,\|\cdot\|_{E},\tau_{\|\cdot\|_{E}})$ a complete Saks space as $\gamma(\|\cdot\|_{E},\tau_{\|\cdot\|_{E}})=\tau_{\|\cdot\|_{E}}$. 
This implies that $(\mathcal{C}v(\Omega,E),\gamma(\|\cdot\|^{E},\tau_{\operatorname{co}}^{E}))$ is also complete since $M_{v}^{E}$ is a 
topological isomorphism. 

(e) Let $\Omega$ be a hemicompact $k_{\R}$-space. Then $\Omega$ is a $k$-space by \cite[Lemma 5.1, p.~884]{mosiman1972} and 
$(\mathcal{C}_{b}(\Omega,E),\gamma(\|\cdot\|_{\infty}^{E},\tau_{\operatorname{co}}^{E}))$ is a Mackey space, 
which is strong if $E$ is a Banach space, by part (f) and \cite[Theorem 3.4, p.~165]{khurana1978c}. 
Let $\Omega$ be a completely metrisable space. 
Then $(\mathcal{C}_{b}(\Omega,E),\gamma(\|\cdot\|_{\infty}^{E},\tau_{\operatorname{co}}^{E}))$ is a Mackey space, 
which is strong if $E$ is a Banach space, by part (f), \cite[Theorem 2, p.~35]{khurana1978b} and \cite[Theorem 3.7, p.~202]{khurana1978a}. Using the topological isomorphism $M_{v}^{E}$ from part (c), we note that both statements remain valid if we 
replace $\mathcal{C}_{b}(\Omega,E)$ by $\mathcal{C}v(\Omega,E)$ and $\|\cdot\|_{\infty}^{E}$ by $\|\cdot\|^{E}$. 
\end{proof}

We remark that $(\mathcal{C}v(\Omega),\|\cdot\|,\tau_{\operatorname{co}})$ is semi-Montel by \cite[3.9 Example (i), p.~10]{kruse2023b} 
if $\Omega$ is discrete. 
In the case $\mathcal{C}_{b}(\Omega,E)$ the statement from \prettyref{cor:cv_mixed} (a)  
that $(\mathcal{C}_{b}(\Omega,E),\|\cdot\|_{\infty}^{E},\tau_{\operatorname{co}}^{E})$ is a Saks space and \prettyref{cor:cv_mixed} (f) 
for $\mathcal{W}_{0}=\mathcal{W}_{\operatorname{usc},0}^{+}(\Omega)$ (see \cite[p.~81--82]{cooper1978}) are contained in 
\cite[II.4.1 Definition, p.~113]{cooper1978} and \cite[II.4.2 Proposition 2), 6), p.~113]{cooper1978} 
(see also \cite[1.1 Remark, p.~844]{fontenot1974}).\footnote{The condition of upper semicontinuity or boundedness for the weights 
$w$ is missing in \cite{fontenot1974} (and \cite{khurana1978a,khurana1978b,khurana1978c}) even though it is contained 
in its reference \cite[Theorem 2.4, p.~316]{sentilles1972} for the proof.} 
In the case $\mathcal{C}_{b}(\Omega)$ \prettyref{cor:cv_mixed} (f) is contained in 
\cite[Proposition 3, p.~590]{cooper1971} for locally compact $\Omega$ and $\mathcal{W}_{0}=\mathcal{C}_{0}^{+}(\Omega)$. 
In the case $\mathcal{C}v(\Omega)$ the statement from \prettyref{cor:cv_mixed} (a) 
that $\gamma(\|\cdot\|,\tau_{\operatorname{co}})=\gamma_{\operatorname{s}}(\|\cdot\|,\tau_{\operatorname{co}})$, 
the inverse of the topological isomorphism $M_{v}^{\K}$ from the proof of part (c) and \prettyref{cor:cv_mixed} (c) 
are contained in \cite[Lemmas A.1, A.4, p.~44]{goldys2022} and \cite[Theorem A.5, p.~44]{goldys2022}. 
Moreover, we note that \prettyref{cor:cv_mixed} (b) does not hold for general $\Omega$ by \cite[Beispiel, p.~232]{kaballo2014}. 
Furthermore, we remark that $(\mathcal{C}v(\Omega,E),\gamma(\|\cdot\|^{E},\tau_{\operatorname{co}}^{E}))$ 
is C-sequential by \cite[Remark 3.19 (a), p.~14]{kruse_schwenninger2022} combined with the topological isomorphism $M_{v}^{E}$ 
if $E=\K$ and $\Omega$ a \emph{Polish space}, i.e.~separably completely metrisable. It is an open question whether this remains 
valid if $E$ is a Banach space. Due to \prettyref{cor:cv_mixed} (e) and \cite[Corollary 7.6, p.~52]{wilansky1981} 
it would be sufficient to prove that $(\mathcal{C}v(\Omega,E),\gamma(\|\cdot\|^{E},\tau_{\operatorname{co}}^{E}))$ is a 
\emph{Mazur space} if $\Omega$ is Polish and $E$ a Banach space, i.e.~that every sequentially $\gamma(\|\cdot\|^{E},\tau_{\operatorname{co}}^{E})$-continuous linear functional 
on $\mathcal{C}v(\Omega,E)$ is already $\gamma(\|\cdot\|^{E},\tau_{\operatorname{co}}^{E})$-continuous.
 
Next, we consider subspaces of $\mathcal{C}v(\Omega,E)$. 
Let $(E,\|\cdot\|_{E})$ be a normed space over $\C$. 
For a non-empty open subset $\Omega$ of a complex locally convex Hausdorff space 
let $\mathcal{H}(\Omega,E)$ be the space of holomorphic functions $f\colon\Omega\to E$, i.e.~the space of G\^{a}teaux-holomorphic 
and continuous functions $f\colon\Omega\to E$ (see \cite[Definition 3.6, p.~152]{dineen1999}),
and for a continuous function $v\colon\Omega\to (0,\infty)$ we set 
\[
\mathcal{H}v(\Omega,E)\coloneqq \{f\in \mathcal{H}(\Omega,E)\;|\;\|f\|^{E}\coloneqq\sup_{z\in\Omega}\|f(z)\|_{E}v(z)<\infty\}.
\]
Further, we define $\mathcal{H}v(\Omega)\coloneqq\mathcal{H}v(\Omega,\C)$ and $\|\cdot\|\coloneqq\|\cdot\|^{\C}$. 
Setting $\Lambda\coloneqq\Omega$, $\mathcal{G}(\Omega,E)\coloneqq\mathcal{H}(\Omega,E)$, $q^{E}\coloneqq 0$ and 
$T^{E}(f)\coloneqq f$ for $f\in\mathcal{H}(\Omega,E)$, we observe that $\FVE=\mathcal{H}v(\Omega,E)$, 
$\tau_{\mathcal{K}_{\Omega}}^{E}=\tau_{\operatorname{co}}^{E}$ on $\mathcal{H}v(\Omega,E)$ and 
conditions (i) and (ii) of \prettyref{prop:relation_weak_strong} are fulfilled. 

\begin{cor}\label{cor:subspace_cont_holo_mixed}
Let $\Omega$ be a non-empty open subset of a locally convex Hausdorff space $X$ over $\C$, $v\colon\Omega\to (0,\infty)$ continuous 
and $(E,\|\cdot\|_{E})$ a normed space over $\C$. Then the following assertions hold. 
\begin{enumerate}
\item $(\mathcal{H}v(\Omega,E),\|\cdot\|^{E},\tau_{\operatorname{co}}^{E})$ is a Saks space. 
\item If $X$ is $k_{\R}$-space and $E$ a Banach space, then $(\mathcal{H}v(\Omega,E),\|\cdot\|^{E},\tau_{\operatorname{co}}^{E})$ 
is complete.
\item If $\Omega$ is hemicompact, then $(\mathcal{H}v(\Omega,E),\|\cdot\|^{E},\tau_{\operatorname{co}}^{E})$ is C-sequential. 
\item If $X$ is a $k$-space and $E=\C$, then $(\mathcal{H}v(\Omega),\|\cdot\|,\tau_{\operatorname{co}})$ is semi-Montel 
and $\gamma(\|\cdot\|,\tau_{\operatorname{co}})=\gamma_{\operatorname{s}}(\|\cdot\|,\tau_{\mathcal{S}})$ for 
$\mathcal{S}\in\{\mathcal{N}_{\Omega},\mathcal{K}_{\Omega}\}$.
\item If $X$ is metrisable or a DFM-space, and $E$ a Banach space, then $\mathcal{H}v(\Omega,E)=\mathcal{H}v(\Omega,E)_{\sigma}$. 
\end{enumerate} 
\end{cor}
\begin{proof}
(a) and (c) follow from \prettyref{thm:standard_saks_space} and our observations above. 

(b) We set $\gamma_{\mathcal{C}v(\Omega,E)}\coloneqq\gamma(\|\cdot\|_{\mathcal{C}v(\Omega,E)},{\tau_{\operatorname{co}}^{E}}_{\mid \mathcal{C}v(\Omega,E)})$ and note that ${\gamma_{\mathcal{C}v(\Omega,E)}}_{\mid \mathcal{H}v(\Omega,E)}\leq \gamma(\|\cdot\|_{\mathcal{H}v(\Omega,E)},{\tau_{\operatorname{co}}^{E}}_{\mid \mathcal{H}v(\Omega,E)})$ 
by \cite[p.~39]{cooper1978}. The space $(\mathcal{C}v(\Omega,E),\gamma_{\mathcal{C}v(\Omega,E)})$ is complete 
by \prettyref{cor:cv_mixed} (c) and $\mathcal{H}v(\Omega,E)$ a closed subspace. This implies statement (b). 

(d) By \cite[3.9 Example (iii), p.~10--11]{kruse2023b} $(\mathcal{H}v(\Omega),\|\cdot\|,\tau_{\operatorname{co}})$ is semi-Montel. 
Furthermore, the observations $\tau_{\mathcal{K}_{\Omega}}=\tau_{\operatorname{co}}$ and 
$\tau_{\mathcal{N}_{\Omega}}\leq \tau_{\operatorname{co}}$ on $\mathcal{H}v(\Omega)$ imply
that $\gamma(\|\cdot\|,\tau_{\operatorname{co}})=\gamma_{\operatorname{s}}(\|\cdot\|,\tau_{\mathcal{S}})$ for 
$\mathcal{S}\in\{\mathcal{N}_{\Omega},\mathcal{K}_{\Omega}\}$ by \prettyref{rem:pre_Saks_unit_ball_char} (c).

(e) This follows from \cite[Example 3.8 (g), p.~159]{dineen1999}.
\end{proof}

Regarding \prettyref{cor:subspace_cont_holo_mixed} (c), we remark that $\Omega$ is hemicompact by 
\cite[Example 2.47, p.~79--81]{dineen1981} if $X$ is a DFM-space. 
\prettyref{thm:standard_saks_space} (e) and \prettyref{cor:subspace_cont_holo_mixed} (d) imply  
\cite[Proposition 3.1, p.~77]{bierstedt_summers1993} where $X=\C^{d}$. 
\prettyref{thm:standard_saks_space} (d) and \prettyref{cor:subspace_cont_holo_mixed} (d) 
imply \cite[4.5 Theorem, p.~875]{mujica1991} where $X$ is a Banach space and $v(z)\coloneqq 1$ for all $z\in\Omega$. 
Further, we remark that $(H^{\infty}(\D),\gamma(\|\cdot\|_{\infty},\tau_{\operatorname{co}}))$ is not a Mackey space by 
\cite[V.2.7 Corollary, p.~235]{cooper1978}.

\begin{rem}
Let $(E,\|\cdot\|)$ be a Banach space over $\C$ and $1\leq p<\infty$. We may also define a strong $E$-valued version of the Hardy 
$H^{p}$ from \prettyref{cor:hardy_bergman_dirichlet}. Let 
\[
H^{p}(E)\coloneqq\Bigl\{f\in\mathcal{H}(\D,E)\;|\;
(\|f\|_{p}^{E})^{p}\coloneqq \sup_{0<r<1}\frac{1}{2\pi}\int_{0}^{2\pi}\|f(r\euler^{\mathrm{i}\theta})\|_{E}^{p}\d\theta<\infty\Bigr\}.
\]
However, in contrast to the case $p=\infty$, we only have the strict inclusion 
$H^{p}(E)\subsetneq H^{p}(E)_{\sigma}$ for $1\leq p<\infty$ by \cite[Corollary 12, p.~359]{freniche2001} 
if $E$ is infinite-dimensional.
\end{rem}

Let us turn to another subspace of $\mathcal{C}v(\Omega,E)$. For a non-empty open set $\Omega\subset\R^{d}$ 
and a normed space $(E,\|\cdot\|_{E})$ over $\K$ we denote by $\mathcal{C}^{\infty}(\Omega,E)$ the space of 
infinitely continuously partially differentiable $E$-valued functions on $\Omega$ and by $(\partial^{\beta})^{E}f$ the $\beta$-th 
partial derivative of $f\in\mathcal{C}^{\infty}(\Omega,E)$ for a multi-index $\beta\in\N_{0}^{d}$. If $E=\K$, 
we set $\mathcal{C}^{\infty}(\Omega)\coloneqq \mathcal{C}^{\infty}(\Omega,\K)$ and 
$\partial^{\beta}f\coloneqq(\partial^{\beta})^{\K}f$ for $f\in\mathcal{C}^{\infty}(\Omega)$ and $\beta\in\N_{0}^{d}$. 
For $\K=\C$ and a polynomial $P$ on $\R^{d}$ with constant complex coefficients we define the linear partial differential operator 
$P(\partial)^{E}\coloneqq P((\partial)^{E})$ on $\mathcal{C}^{\infty}(\Omega,E)$ in the usual way and its kernel
\[
\mathcal{C}_{P}(\Omega,E)\coloneqq\{f\in\mathcal{C}^{\infty}(\Omega,E)\;|\;f\in\ker P(\partial)^{E}\}.
\]
For a continuous function $v\colon\Omega\to(0,\infty)$ we define the weighted kernel
\[
  \mathcal{C}_{P}v(\Omega,E)
\coloneqq \{f\in\mathcal{C}_{P}(\Omega,E)\;|\;\|f\|^{E}\coloneqq\sup_{x\in\Omega}\|f(x)\|_{E}v(x)<\infty\}.
\]
Further, we define $\mathcal{C}_{P}v(\Omega)\coloneqq\mathcal{C}_{P}v(\Omega,\C)$ and $\|\cdot\|\coloneqq\|\cdot\|^{\C}$. 
Setting $\Lambda\coloneqq\Omega$, $\mathcal{G}(\Omega,E)\coloneqq\mathcal{C}_{P}(\Omega,E)$, $q^{E}\coloneqq 0$ and 
$T^{E}(f)\coloneqq f$ for $f\in\mathcal{C}_{P}(\Omega,E)$, we observe that $\FVE=\mathcal{C}_{P}v(\Omega,E)$, 
$\tau_{\mathcal{K}_{\Omega}}^{E}=\tau_{\operatorname{co}}^{E}$ on $\mathcal{C}_{P}v(\Omega,E)$ and 
conditions (i) and (ii) of \prettyref{prop:relation_weak_strong} are fulfilled. 

\begin{cor}\label{cor:subspace_cont_hypo_mixed}
Let $\Omega\subset\R^{d}$ be non-empty and open, $v\colon\Omega\to (0,\infty)$ continuous, $(E,\|\cdot\|_{E})$ a normed space over $\C$ 
and $P(\partial)^{E}$ a linear partial differential operator such that $P(\partial)^{\C}$ is hypoelliptic. 
Then the following assertions hold. 
\begin{enumerate}
\item $(\mathcal{C}_{P}v(\Omega,E),\|\cdot\|^{E},\tau_{\operatorname{co}}^{E})$ is a C-sequential Saks space.
\item If $E=\C$, then $(\mathcal{C}_{P}v(\Omega),\|\cdot\|,\tau_{\operatorname{co}})$ is semi-Montel and 
$\gamma(\|\cdot\|,\tau_{\operatorname{co}})=\gamma_{\operatorname{s}}(\|\cdot\|,\tau_{\mathcal{S}})$ for 
$\mathcal{S}\in\{\mathcal{N}_{\Omega},\mathcal{K}_{\Omega}\}$.
\item If $E$ is a Banach space, then $\mathcal{C}_{P}v(\Omega,E)=\mathcal{C}_{P}v(\Omega,E)_{\sigma}$ and the Saks space 
$(\mathcal{C}_{P}v(\Omega,E),\|\cdot\|^{E},\tau_{\operatorname{co}}^{E})$ is complete.
\end{enumerate}
\end{cor}
\begin{proof} 
(a) This follows from \prettyref{thm:standard_saks_space}, our observations above and the fact that open subsets of $\R^{d}$ 
are hemicompact. 

(b) By \cite[3.9 Example (ii), p.~10]{kruse2023b} $(\mathcal{C}_{P}v(\Omega),\|\cdot\|,\tau_{\operatorname{co}})$ is semi-Montel. 
The rest of the proof is analogous to the proof of \prettyref{cor:subspace_cont_holo_mixed} (d).

(c) It holds $\mathcal{C}_{P}v(\Omega,E)=\mathcal{C}_{P}v(\Omega,E)_{\sigma}$ 
by the weak-strong principle \cite[Theorem 9, p.~232]{bonet2007} and Mackey's theorem. 
Due to part (b), $\tau_{\operatorname{p}}\leq\tau_{\mathcal{K}_{\Omega}}$, 
$\tau_{\mathcal{K}_{\Omega}}^{E}=\tau_{\operatorname{co}}^{E}$ and \prettyref{prop:relation_weak_strong} (c) 
we get that the Saks space $(\mathcal{C}_{P}v(\Omega,E),\|\cdot\|^{E},\tau_{\operatorname{co}}^{E})$ is complete.
\end{proof}

\begin{rem}\label{rem:saks_space_weighted_cont}
Let $(X,\|\cdot\|,\tau)$ be a Saks space, $Y$ a linear subspace of $X$. 
Then the subtriple $(Y,\|\cdot\|_{\mid Y},\tau_{\mid Y})$ is also a Saks space and 
$\gamma(\|\cdot\|,\tau)_{\mid Y}\leq \gamma(\|\cdot\|_{\mid Y},\tau_{\mid Y})$ by \cite[p.~39]{cooper1978}. 
In general, $\gamma(\|\cdot\|,\tau)_{\mid Y}=\gamma(\|\cdot\|_{\mid Y},\tau_{\mid Y})$ does not hold 
by \cite[p.~133]{alexiewicz1958}. However, suppose that $(Y,\|\cdot\|_{\mid Y},\tau_{\mid Y})$ fulfils condition (i) or 
(ii) of \prettyref{rem:mixed=submixed} (b) so that 
$\gamma(\|\cdot\|_{\mid Y},\tau_{\mid Y})=\gamma_{\operatorname{s}}(\|\cdot\|_{\mid Y},\tau_{\mid Y})$. 
Then it follows from \cite[I.4.6 Lemma, p.~44]{cooper1978} 
that $\gamma(\|\cdot\|,\tau)_{\mid Y}=\gamma(\|\cdot\|_{\mid Y},\tau_{\mid Y})$. 
This observation is an alternative way to derive $\gamma(\|\cdot\|,\tau_{\operatorname{co}})=\gamma_{\operatorname{s}}(\|\cdot\|,\tau_{\mathcal{K}_{\Omega}})$ in \prettyref{cor:subspace_cont_holo_mixed} (d) and \prettyref{cor:subspace_cont_hypo_mixed} (b) 
since condition (ii) of \prettyref{rem:mixed=submixed} (b) is fulfilled for the subtriple.\footnote{Since \cite[I.4.6 Lemma, p.~44]{cooper1978} needs that condition (i) or (ii) of \prettyref{rem:mixed=submixed} (b) is fulfilled for 
the subtriple $(Y,\|\cdot\|_{\mid Y},\tau_{\mid Y})$ and not for the triple $(X,\|\cdot\|,\tau)$, the proof of \cite[5.~Proposition, p.~291]{prieto1992} seems to be doubtful where $X=\mathcal{C}_{b}(U)$, $U$ an open connected subset of 
a complex Banach space, $\|\cdot\|=\|\cdot\|_{\infty}$, $\tau=\tau_{b}$ is the topology of uniform convergence on $U$-bounded subsets of $U$ and $Y=H^{\infty}(U)$. With the definition of the strict topology $\beta$ on 
$\mathcal{C}_{b}(U)$ in \cite[p.290]{prieto1992} it is shown that $\beta=\gamma(\|\cdot\|_{\infty},\tau_{b})$ 
in \cite[3.~Theorem, p.~291]{prieto1992}. Then \cite[5.~Proposition, p.~291]{prieto1992} says that 
$\beta_{\mid H^{\infty}(U)}=\gamma(\|\cdot\|_{\mid H^{\infty}(U)},{\tau_{b}}_{\mid H^{\infty}(U)})$. 
However, in its proof it is only shown that the triple $(\mathcal{C}_{b}(U),\|\cdot\|_{\infty},\tau_{b})$ 
fulfils condition (i) of \prettyref{rem:mixed=submixed} (b), not the subtriple 
$(H^{\infty}(U),\|\cdot\|_{\mid H^{\infty}(U)},{\tau_{b}}_{\mid H^{\infty}(U)})$. 
At least if $U$ is a subset of a finite-dimensional complex Banach space, then one can use $\tau_{b}=\tau_{\operatorname{co}}$ which gives that condition (ii) of \prettyref{rem:mixed=submixed} (b) is fulfilled for the subtriple in this case.}
\end{rem}
 
For a normed space $(E,\|\cdot\|_{E})$ over $\C$ and a continuous function $v\colon\D\to(0,\infty)$ with 
$\D=\{z\in\C\;|\;|z|<1\}$ we define the Bloch type space 
\[
\mathcal{B}v(\D,E)\coloneqq\{f\in\mathcal{H}(\D,E)\;|\;\|f\|^{E}\coloneqq \|f(0)\|_{E}+\sup_{z\in\D}\|\partial_{\C}^{E}f(z)\|_{E}v(z)<\infty\}
\]
where 
\[
\partial_{\C}^{E}f(z)\coloneqq\lim_{\substack{h\to 0\\h\in\C,h\neq 0}}\frac{f(z+h)-f(z)}{h},\quad z\in\D,
\]
Further, we define $\mathcal{B}v(\D)\coloneqq\mathcal{B}v(\D,\C)$ and $\|\cdot\|\coloneqq\|\cdot\|^{\C}$. 
Setting $\Lambda\coloneqq\D$, $\mathcal{G}(\D,E)\coloneqq\mathcal{H}(\D,E)$, $q^{E}(f)\coloneqq \|f(0)\|_{E}$ and 
$T^{E}(f)\coloneqq \partial_{\C}^{E}f(z)$ for $f\in\mathcal{H}(\D,E)$, we observe that $\FVE=\mathcal{B}v(\D,E)$ 
and conditions (i) and (ii) of \prettyref{prop:relation_weak_strong} are fulfilled. 

\begin{cor}\label{cor:bloch_mixed}
Let $v\colon\D\to(0,\infty)$ be continuous and $(E,\|\cdot\|_{E})$ a normed space over $\C$. Then the following assertions hold. 
\begin{enumerate}
\item $(\mathcal{B}v(\D,E),\|\cdot\|^{E},\tau_{\operatorname{co}}^{E})$ is a C-sequential Saks space and 
$\tau_{\operatorname{co}}^{E}=\tau_{\mathcal{K}_{\D}}^{E}$ on $\mathcal{B}v(\D,E)$.
\item If $E=\C$, then $(\mathcal{B}v(\D),\|\cdot\|,\tau_{\operatorname{co}})$ is semi-Montel and 
$\gamma(\|\cdot\|,\tau_{\operatorname{co}})=\gamma_{\operatorname{s}}(\|\cdot\|,\tau_{\mathcal{S}})$ for 
$\mathcal{S}\in\{\mathcal{N}_{\D},\mathcal{K}_{\D}\}$.
\item If $E$ is a Banach space, then $\mathcal{B}v(\D,E)=\mathcal{B}v(\D,E)_{\sigma}$ and the Saks space 
$(\mathcal{B}v(\D,E),\|\cdot\|^{E},\tau_{\operatorname{co}}^{E})$ is complete.
\end{enumerate}
\end{cor}
\begin{proof}
(a) For every $0<r<1$ we have 
\begin{align*}
     \max_{|z|\leq r}\|f(z)\|_{E}
&\leq \|f(0)\|_{E} +\max_{|z|\leq r}\|\int_{0}^{z}\partial_{\C}^{E}f(\zeta)\d\zeta\|_{E}\\    
&\leq \Bigl(1+\frac{r}{\min_{|\zeta|\leq r}v(\zeta)}\Bigr)
     \bigr(\|f(0)\|_{E}+\sup_{|\zeta|\leq r}\|\partial_{\C}^{E}(\zeta)\|_{E}v(\zeta)\bigl)
\end{align*}
for all $f\in \mathcal{B}v(\D,E)$ where the integral in the estimate above is a Bochner integral 
(cf.~\cite[Corollary 3.8, p.~9--10]{kruse2019_3} for the case $E=\C$), 
and for every $0<s<r<1$ 
\begin{align*}
      \|f(0)\|_{E}+\max_{|z|\leq s}\|\partial_{\C}^{E}f(z)\|_{E}v(z)
&\leq \|f(0)\|_{E}+\frac{1}{r}\max_{|z|\leq s}v(z)\max_{|\zeta|\leq r}\|f(\zeta)\|_{E}\\
&\leq \Bigl(1+\frac{1}{r}\max_{|z|\leq s}v(z)\Bigr)\max_{|\zeta|\leq r}\|f(\zeta)\|_{E}
\end{align*}
for all $f\in \mathcal{B}v(\D,E)$ by Cauchy's inequality, which proves 
$\tau_{\operatorname{co}}^{E}=\tau_{\mathcal{K}_{\D}}^{E}$ on $\mathcal{B}v(\D,E)$. 
The rest of part (a) follows from \prettyref{thm:standard_saks_space}, our observations above and the fact that $\D$ 
is hemicompact. 

(b) By \cite[3.9 Example (iv), p.~11]{kruse2023b} $(\mathcal{B}v(\D),\|\cdot\|,\tau_{\operatorname{co}})$ is semi-Montel. 
The rest of the proof is analogous to the proof of \prettyref{cor:subspace_cont_holo_mixed} (d).

(c) It holds $\mathcal{B}v(\D,E)=\mathcal{B}v(\D,E)_{\sigma}$ 
by the weak-strong principle \cite[Theorem 9, p.~232]{bonet2007} and Mackey's theorem. 
Due to part (b), $\tau_{\operatorname{p}}\leq\tau_{\mathcal{K}_{\D}}$, 
$\tau_{\mathcal{K}_{\D}}^{E}=\tau_{\operatorname{co}}^{E}$ and \prettyref{prop:relation_weak_strong} (c) 
we get that the Saks space $(\mathcal{B}v(\D,E),\|\cdot\|^{E},\tau_{\operatorname{co}}^{E})$ is complete.
\end{proof}

For a normed space $(E,\|\cdot\|_{E})$ over $\K$ and a metric space $(\Omega,\d)$ with a base point denoted by $\mathbf{0}$, 
i.e.~a \emph{pointed metric space} in the sense of \cite[p.~1]{weaver2018}, 
we define the space of $E$-valued Lipschitz continuous on $(\Omega,\d)$ that vanish at $\mathbf{0}$ by 
\[
\mathrm{Lip}_{0}(\Omega,E)
\coloneqq\Bigl\{f\colon\Omega\to E\;|\;f(\mathbf{0})=0\text{ and }\|f\|^{E}\coloneqq\sup_{\substack{x,y\in\Omega\\x\neq y}}
\frac{\|f(x)-f(y)\|_{E}}{\d(x,y)}<\infty\Bigr\}.
\]
Further, we define $\mathrm{Lip}_{0}(\Omega)\coloneqq\mathrm{Lip}_{0}(\Omega,\K)$ and $\|\cdot\|\coloneqq\|\cdot\|^{\K}$. 
Setting $\Lambda\coloneqq\Omega_{\operatorname{wd}}\coloneqq\{(x,y)\in\Omega^2\;|\;x\neq y\}$, $v\colon\Lambda\to (0,\infty)$, $v(x,y)\coloneqq \frac{1}{\d(x,y)}$, 
$\GE\coloneqq\{f\colon\Omega\to E\;|\;f(\mathbf{0})=0\}$, 
$q^{E}\coloneqq 0$ and $T^{E}(f)(x,y)\coloneqq f(x)-f(y)$ for $(x,y)\in\Lambda$ and $f\in\GE$, we observe that 
$\FVE=\mathrm{Lip}_{0}(\Omega,E)$ and conditions (i) and (ii) of \prettyref{prop:relation_weak_strong} are fulfilled. 

\begin{cor}\label{cor:lipschitz_vanish_mixed}
Let $(\Omega,\d)$ be a pointed metric space and $(E,\|\cdot\|_{E})$ a normed space over $\K$. 
Then the following assertions hold. 
\begin{enumerate}
\item $(\mathrm{Lip}_{0}(\Omega,E),\|\cdot\|^{E},\tau_{\operatorname{co}}^{E})$ is a Saks space, 
$\tau_{\operatorname{co}}^{E}=\tau_{\mathcal{K}_{\Omega_{\operatorname{wd}}}}^{E}$ on $\|\cdot\|^{E}$-bounded sets, 
$\gamma(\|\cdot\|^{E},\tau_{\operatorname{co}}^{E})=\gamma(\|\cdot\|^{E},\tau_{\mathcal{K}_{\Omega_{\operatorname{wd}}}}^{E})$ 
and $\mathrm{Lip}_{0}(\Omega,E)=\mathrm{Lip}_{0}(\Omega,E)_{\sigma}$.
\item If $E$ is a Banach space, then $(\mathrm{Lip}_{0}(\Omega,E),\|\cdot\|^{E},\tau_{\operatorname{co}}^{E})$ 
is complete.
\item If $\Omega$ is hemicompact, then $(\mathrm{Lip}_{0}(\Omega,E),\|\cdot\|^{E},\tau_{\operatorname{co}}^{E})$ 
is C-sequential.
\item If $E=\K$, then $(\mathrm{Lip}_{0}(\Omega),\|\cdot\|,\tau_{\operatorname{co}})$ is semi-Montel and 
$\gamma(\|\cdot\|,\tau_{\operatorname{co}})=\gamma_{\operatorname{s}}(\|\cdot\|,\tau_{\mathcal{S}})$ for 
$\mathcal{S}\in\{\mathcal{N}_{\Omega_{\operatorname{wd}}},\mathcal{K}_{\Omega_{\operatorname{wd}}}\}$.
\end{enumerate}
\end{cor}
\begin{proof}
(a) First, we note that for compact $K\subset\Omega_{\operatorname{wd}}$ the projections $\pi_{1}(K)$ and $\pi_{2}(K)$ on the first and 
second component, respectively, are compact in $\Omega$ and 
\[
q_{K}^{E}(f)=\sup_{(x,y)\in K}\frac{\|f(x)-f(y)\|_{E}}{\d(x,y)}
\leq 2\max_{(x,y)\in K}\frac{1}{{\d(x,y)}}\sup_{x\in (\pi_{1}(K)\cup \pi_{2}(K))}\|f(x)\|_{E}
\]
for all $f\in\mathrm{Lip}_{0}(\Omega,E)$, which implies 
$\tau_{\mathcal{K}_{\Omega_{\operatorname{wd}}}}^{E}\leq\tau_{\operatorname{co}}^{E}$ on $\mathrm{Lip}_{0}(\Omega,E)$. 
On the other hand, for every $\varepsilon >0$ and compact $K\subset\Omega$ we have with 
$\mathbb{B}_{\varepsilon}\coloneqq\{x\in\Omega\;|\;\d(x,\mathbf{0})<\varepsilon\}$ that 
\begin{align*}
\sup_{x\in K}\|f(x)\|_{E}
&=\sup_{x\in K,x\neq\mathbf{0}}\frac{\|f(x)-f(\mathbf{0})\|_{E}}{\d(x,\mathbf{0})}\d(x,\mathbf{0})\\
&\leq \max_{x\in K}\d(x,\mathbf{0})
\sup_{x\in K\setminus\mathbb{B}_{\varepsilon}}\frac{\|f(x)-f(\mathbf{0})\|_{E}}{\d(x,\mathbf{0})}
+\varepsilon \sup_{x\in\mathbb{B}_{\varepsilon},x\neq\mathbf{0}}\frac{\|f(x)-f(\mathbf{0})\|_{E}}{\d(x,\mathbf{0})}\\
&\leq \max_{x\in K}\d(x,\mathbf{0})q_{(K\setminus\mathbb{B}_{\varepsilon})\times\{\mathbf{0}\}}^{E}(f)+\varepsilon \|f\|^{E}
\end{align*}
for all $f\in\mathrm{Lip}_{0}(\Omega,E)$. Thus the topologies $\tau_{\mathcal{K}_{\Omega_{\operatorname{wd}}}}^{E}$ and 
$\tau_{\operatorname{co}}^{E}$ coincide on $\|\cdot\|^{E}$-bounded sets. 
Due to \cite[I.3.1 Lemma, p.~27]{cooper1978} and \prettyref{thm:standard_saks_space} (a) 
this yields that $(\mathrm{Lip}_{0}(\Omega,E),\|\cdot\|^{E},\tau_{\operatorname{co}}^{E})$ is a Saks space. In addition, 
we deduce that
$\gamma(\|\cdot\|^{E},\tau_{\operatorname{co}}^{E})=\gamma(\|\cdot\|^{E},\tau_{\mathcal{K}_{\Omega_{\operatorname{wd}}}}^{E})$ 
by the definition of the mixed topology. Moreover, it holds that $\mathrm{Lip}_{0}(\Omega,E)=\mathrm{Lip}_{0}(\Omega,E)_{\sigma}$ 
by Mackey's theorem. 

(b) This follows from parts (a) and (d), $\tau_{\operatorname{p}}\leq\tau_{\mathcal{K}_{\Omega_{\operatorname{wd}}}}$ 
and \prettyref{prop:relation_weak_strong} (c).

(c) If $\Omega$ is hemicompact, then $\tau_{\operatorname{co}}^{E}$ is metrisable 
and so $(\mathrm{Lip}_{0}(\Omega,E),\|\cdot\|^{E},\tau_{\operatorname{co}}^{E})$ C-sequential by 
\cite[Proposition 5.7, p.~2681--2682]{kruse_meichnser_seifert2018}.

(d) By \cite[3.9 Example (v), p.~11]{kruse2023b} $(\mathrm{Lip}_{0}(\Omega),\|\cdot\|,\tau_{\operatorname{co}})$ is semi-Montel. 
Furthermore, the observations 
$\tau_{\mathcal{N}_{\Omega_{\operatorname{wd}}}}\leq\tau_{\mathcal{K}_{\Omega_{\operatorname{wd}}}}\leq\tau_{\operatorname{co}}$ on $\mathrm{Lip}_{0}(\Omega)$ imply
that $\gamma(\|\cdot\|,\tau_{\operatorname{co}})=\gamma_{\operatorname{s}}(\|\cdot\|,\tau_{\mathcal{S}})$ for 
$\mathcal{S}\in\{\mathcal{N}_{\Omega_{\operatorname{wd}}},\mathcal{K}_{\Omega_{\operatorname{wd}}}\}$ by \prettyref{rem:pre_Saks_unit_ball_char} (c).
\end{proof}

Regarding \prettyref{cor:lipschitz_vanish_mixed} (c), we remark that a metric space is hemicompact if and only if it 
is separable and locally compact by \cite[Exercises 3.4.E (a), (c), p.~165]{engelking1989}, 
\cite[Exercises 3.8.C (b), p.~194--195]{engelking1989} and \cite[16.11 Theorem, p.~112]{willard1970}. 
The statement that $(\mathrm{Lip}_{0}(\Omega),\|\cdot\|,\tau_{\operatorname{co}})$ 
is a complete semi-Montel Saks space from \prettyref{cor:lipschitz_vanish_mixed} (b) and (d) for $E=\K$ 
is already contained in \cite[Theorem 2.1 (7), p.~642]{vargas2018}. 
\prettyref{cor:lipschitz_vanish_mixed} (c) and \prettyref{thm:standard_saks_space} (d) imply 
\cite[Theorem 3.4, p.~647]{vargas2018}. 
Further, \prettyref{cor:lipschitz_vanish_mixed} (c) and \prettyref{thm:standard_saks_space} (e) imply 
\cite[Theorem 3.3, p.~645]{vargas2018} where $\Omega$ is compact and $\mathcal{W}_{0}=\mathcal{C}_{0}^{+}(\Omega_{\operatorname{wd}})$. 

For a normed space $(E,\|\cdot\|_{E})$ over $\K$ and $k\in\N_{0}$ we denote by $\mathcal{C}^{k}(\Omega,E)$ the space of 
$k$-times continuously partially differentiable $E$-valued functions on a non-empty open bounded set $\Omega\subset\R^{d}$. 
We define the space of $k$-times continuously partially differentiable $E$-valued functions on $\Omega$ whose partial derivatives 
up to order $k$ are continuously extendable to the boundary of $\Omega$ by 
\[
 \mathcal{C}^{k}(\overline{\Omega},E)\coloneqq\{f\in\mathcal{C}^{k}(\Omega,E)\;|\;(\partial^{\beta})^{E}f\;
 \text{cont.\ extendable on}\;\overline{\Omega}\;\text{for all}\;\beta\in\N^{d}_{0},\,|\beta|\leq k\}
\]
which we equip with the norm given by 
\[
 |f|_{\mathcal{C}^{k}(\overline{\Omega})}^{E}
\coloneqq \sup_{\substack{x\in \Omega\\ \beta\in\N^{d}_{0}, |\beta|\leq k}}\|(\partial^{\beta})^{E}f(x)\|_{E} 
=\sup_{\substack{x\in \overline{\Omega}\\ \beta\in\N^{d}_{0}, |\beta|\leq k}}\|(\partial^{\beta})^{E}f(x)\|_{E}, 
 \quad f\in\mathcal{C}^{k}(\overline{\Omega},E),
\]
where we use the same symbol for the unique continuous extension of $(\partial^{\beta})^{E}f$ to $\overline{\Omega}$. 
The space of functions in $\mathcal{C}^{k}(\overline{\Omega},E)$ such 
that all its $k$-th partial derivatives are $\alpha$-H\"older continuous with $0<\alpha\leq 1$ is given by  
\[
\mathcal{C}^{k,\alpha}(\overline{\Omega},E)\coloneqq 
\bigl\{f\in\mathcal{C}^{k}(\overline{\Omega},E)\;|\;
\|f\|^{E}<\infty\bigr\}
\]
where
\[
\|f\|^{E}\coloneqq |f|_{\mathcal{C}^{k}(\overline{\Omega})}^{E}
+\sup_{\beta\in\N^{d}_{0}, |\beta|=k}\sup_{\substack{x,y\in\Omega\\x\neq y}}
\frac{\|(\partial^{\beta})^{E}f(x)-(\partial^{\beta})^{E}f(y)\|_{E}}{|x-y|^{\alpha}}.
\]
Further, we define $\mathcal{C}^{k}(\overline{\Omega})\coloneqq \mathcal{C}^{k}(\overline{\Omega},\K)$, 
$\mathcal{C}^{k,\alpha}(\overline{\Omega})\coloneqq\mathcal{C}^{k,\alpha}(\overline{\Omega},\K)$ and 
$|\cdot|_{\mathcal{C}^{k}(\overline{\Omega})}\coloneqq |\cdot|_{\mathcal{C}^{k}(\overline{\Omega})}^{\K}$ as well as 
$\|\cdot\|\coloneqq\|\cdot\|^{\K}$. 

Let $E$ be Banach space. Then for every $\beta\in\N_{0}^{d}$ with $|\beta|=k$ and $f\in\mathcal{C}^{k,\alpha}(\overline{\Omega},E)$ 
the unique continuous extension of the partial derivative $(\partial^{\beta})^{E}f$ to 
$\overline{\Omega}$ is $\alpha$-H\"older continuous and the extension has the same H\"older constant 
by \cite[Proposition 1.6, p.~5]{weaver2018} and \cite[Proposition 2.50, p.~66]{weaver2018}, i.e.
\[
\sup_{\substack{x,y\in\Omega\\x\neq y}}\frac{\|(\partial^{\beta})^{E}f(x)-(\partial^{\beta})^{E}f(y)\|_{E}}{|x-y|^{\alpha}}
=\sup_{\substack{x,y\in\overline{\Omega}\\x\neq y}}
\frac{\|(\partial^{\beta})^{E}f(x)-(\partial^{\beta})^{E}f(y)\|_{E}}{|x-y|^{\alpha}}.
\]
Setting $\overline{\Omega}_{\operatorname{wd}}\coloneqq\{(x,y)\in\overline{\Omega}^2\;|\;x\neq y\}$,
$\Lambda\coloneqq\{\beta\in\N^{d}_{0}\;|\;|\beta|=k\}\times\overline{\Omega}_{\operatorname{wd}}$, 
$v\colon\Lambda\to (0,\infty)$, $v(\beta,x,y)\coloneqq \frac{1}{|x-y|^{\alpha}}$, 
$\GE\coloneqq\mathcal{C}^{k,\alpha}(\overline{\Omega},E)$, 
$q^{E}\coloneqq |\cdot|_{\mathcal{C}^{k}(\overline{\Omega})}^{E}$ and 
$T^{E}(f)(\beta,x,y)\coloneqq (\partial^{\beta})^{E}f(x)-(\partial^{\beta})^{E}f(y)$ for $(\beta,x,y)\in\Lambda$ and $f\in\GE$, 
we remark that 
$\FVE=\mathcal{C}^{k,\alpha}(\overline{\Omega},E)$ and conditions (i) and (ii) of \prettyref{prop:relation_weak_strong} are fulfilled. 
Furthermore, we denote by $\mathcal{N}_{k,\overline{\Omega}_{\operatorname{wd}}}$ the family of subsets of $\Lambda$ 
of the form $\{\beta\in\N^{d}_{0}\;|\;|\beta|=k\}\times N$ for $N\in\mathcal{N}_{\overline{\Omega}_{\operatorname{wd}}}$.
Similarly, we denote by $\mathcal{K}_{k,\overline{\Omega}_{\operatorname{wd}}}$ the family of subsets of $\Lambda$ of the form $\{\beta\in\N^{d}_{0}\;|\;|\beta|=k\}\times K$ for $K\in\mathcal{K}_{\overline{\Omega}_{\operatorname{wd}}}$. 
Since $\{\beta\in\N^{d}_{0}\;|\;|\beta|=k\}$ is a finite set, we have 
$\tau_{\mathcal{N}_{\Lambda}}^{E}=\tau_{\mathcal{N}_{k,\overline{\Omega}_{\operatorname{wd}}}}^{E}$ and 
$\tau_{\mathcal{K}_{\Lambda}}^{E}=\tau_{\mathcal{K}_{k,\overline{\Omega}_{\operatorname{wd}}}}^{E}$. 

\begin{cor}\label{cor:hoelder_diff__mixed}
Let $\Omega\subset\R^{d}$ be a non-empty open bounded set, $k\in\N_{0}$, $0<\alpha\leq 1$ and $(E,\|\cdot\|_{E})$ a Banach space 
over $\K$. Then the following assertions hold. 
\begin{enumerate}
\item $(\mathcal{C}^{k,\alpha}(\overline{\Omega},E),\|\cdot\|^{E},\tau_{|\cdot|_{\mathcal{C}^{k}(\overline{\Omega})}^{E}})$ 
is a C-sequential Saks space, $\tau_{|\cdot|_{\mathcal{C}^{k}(\overline{\Omega})}^{E}}=\tau_{\mathcal{K}_{k,\overline{\Omega}_{\operatorname{wd}}}}^{E}$ 
on $\|\cdot\|^{E}$-bounded sets and 
$\gamma(\|\cdot\|^{E},\tau_{|\cdot|_{\mathcal{C}^{k}(\overline{\Omega})}^{E}})=\gamma(\|\cdot\|^{E},\tau_{\mathcal{K}_{k,\overline{\Omega}_{\operatorname{wd}}}}^{E})$.
\end{enumerate}
Suppose for (b)--(c) that $\Omega$ has Lipschitz boundary if $k\geq 1$.
\begin{enumerate}
\item[(b)] If $E=\K$, then $(\mathcal{C}^{k,\alpha}(\overline{\Omega}),\|\cdot\|,\tau_{|\cdot|_{\mathcal{C}^{k}(\overline{\Omega})}})$ 
is semi-Montel and 
$\gamma(\|\cdot\|,\tau_{|\cdot|_{\mathcal{C}^{k}(\overline{\Omega})}})=\gamma_{\operatorname{s}}(\|\cdot\|,\tau_{\mathcal{S}})$ for 
$\mathcal{S}\in\{\mathcal{N}_{k,\overline{\Omega}_{\operatorname{wd}}},\mathcal{K}_{k,\overline{\Omega}_{\operatorname{wd}}}\}$.
\item[(c)] $\mathcal{C}^{k,\alpha}(\overline{\Omega},E)=\mathcal{C}^{k,\alpha}(\overline{\Omega},E)_{\sigma}$ and 
$(\mathcal{C}^{k,\alpha}(\overline{\Omega},E),\|\cdot\|^{E},\tau_{|\cdot|_{\mathcal{C}^{k}(\overline{\Omega})}^{E}})$ is complete. 
\end{enumerate}
\end{cor}
\begin{proof}
(a) We set $M\coloneqq \{\beta\in\N^{d}_{0}\;|\;|\beta|=k\}$ and note that 
for compact $K\subset\overline{\Omega}_{\operatorname{wd}}$ the projections $\pi_{1}(K)$ and $\pi_{2}(K)$ 
on the first and second component, respectively, are compact in $\overline{\Omega}$ and 
\begin{align*}
q_{M\times K}^{E}(f)
&=\sup_{\substack{(x,y)\in K \\ \beta\in\N^{d}_{0}, |\beta|= k}}
|f|_{\mathcal{C}^{k}(\overline{\Omega})}^{E}+\frac{\|(\partial^{\beta})^{E}f(x)-(\partial^{\beta})^{E}f(y)\|_{E}}{|x-y|^{\alpha}}\\
&\leq |f|_{\mathcal{C}^{k}(\overline{\Omega})}^{E}+2\max_{(x,y)\in K}\frac{1}{|x-y|^{\alpha}}
\sup_{\substack{x\in (\pi_{1}(K)\cup \pi_{2}(K))\\ \beta\in\N^{d}_{0}, |\beta|= k}}\|(\partial^{\beta})^{E}f(x)\|_{E}\\
&\leq \Bigl(1+2\max_{(x,y)\in K}\frac{1}{|x-y|^{\alpha}}\Bigr)|f|_{\mathcal{C}^{k}(\overline{\Omega})}^{E}
\end{align*}
for all $f\in \mathcal{C}^{k,\alpha}(\overline{\Omega})$, which implies 
$\tau_{\mathcal{K}_{k,\overline{\Omega}_{\operatorname{wd}}}}^{E}\leq\tau_{|\cdot|_{\mathcal{C}^{k}(\overline{\Omega})}^{E}}$. 
On the other hand, fix some $y_{0}\in\Omega$. For every $\varepsilon >0$ we have with 
$\mathbb{B}_{\varepsilon}\coloneqq\{x\in\overline{\Omega}\;|\;|x-y_{0}|^{\alpha}<\varepsilon\}$ that 
\begin{align*}
 |f|_{\mathcal{C}^{k}(\overline{\Omega})}^{E}
&\leq \sup_{\beta\in\N_{0}^{d},|\beta|\leq k}\|(\partial^{\beta})^{E}f(y_{0})\|_{E}
 +\sup_{\substack{x\in\Omega, x\neq y_{0} \\ \beta\in\N_{0}^{d},|\beta|\leq k}}\|(\partial^{\beta})^{E}f(x)-(\partial^{\beta})^{E}f(y_{0})\|_{E}\\
&\leq |f|_{\mathcal{C}^{k}(\overline{\Omega})}^{E}
 +\sup_{\substack{x\in\overline{\Omega}, x\neq y_{0} \\ \beta\in\N_{0}^{d},|\beta|\leq k}}
 \frac{\|(\partial^{\beta})^{E}f(x)-(\partial^{\beta})^{E}f(y_{0})\|_{E}}{|x-y_{0}|^{\alpha}}|x-y_{0}|^{\alpha} \\
&\leq |f|_{\mathcal{C}^{k}(\overline{\Omega})}^{E}
 +\max_{x\in\overline{\Omega}}|x-y_{0}|^{\alpha}
 \sup_{\substack{x\in\overline{\Omega}\setminus\mathbb{B}_{\varepsilon}\\ \beta\in\N_{0}^{d},|\beta|\leq k}}
 \frac{\|(\partial^{\beta})^{E}f(x)-(\partial^{\beta})^{E}f(y_{0})\|_{E}}{|x-y_{0}|^{\alpha}}\\
&\phantom{\leq}+\varepsilon \sup_{\substack{x\in\mathbb{B}_{\varepsilon}, x\neq y_{0} \\ \beta\in\N_{0}^{d},|\beta|\leq k}}
 \frac{\|(\partial^{\beta})^{E}f(x)-(\partial^{\beta})^{E}f(y_{0})\|_{E}}{|x-y_{0}|^{\alpha}}\\
&\leq \Bigl(1+\max_{x\in\overline{\Omega}}|x-y_{0}|^{\alpha}\Bigr)
 q_{M\times((\overline{\Omega}\setminus\mathbb{B}_{\varepsilon})\times\{y_{0}\})}^{E}(f)+\varepsilon \|f\|^{E}
\end{align*}
for all $f\in\mathcal{C}^{k,\alpha}(\overline{\Omega},E)$. 
Thus the topologies $\tau_{\mathcal{K}_{k,\overline{\Omega}_{\operatorname{wd}}}}^{E}$ and 
$\tau_{|\cdot|_{\mathcal{C}^{k}(\overline{\Omega})}^{E}}$ coincide on $\|\cdot\|^{E}$-bounded sets. 
Due to \cite[I.3.1 Lemma, p.~27]{cooper1978} and \prettyref{thm:standard_saks_space} (a) 
this yields that $(\mathcal{C}^{k,\alpha}(\overline{\Omega},E),\|\cdot\|^{E},\tau_{|\cdot|_{\mathcal{C}^{k}(\overline{\Omega})}^{E}})$ 
is a Saks space. Furthermore, we deduce that
$\gamma(\|\cdot\|^{E},\tau_{|\cdot|_{\mathcal{C}^{k}(\overline{\Omega})}^{E}})=\gamma(\|\cdot\|^{E},\tau_{\mathcal{K}_{k,\overline{\Omega}_{\operatorname{wd}}}}^{E})$ by the definition of the mixed topology. 
Since $\tau_{|\cdot|_{\mathcal{C}^{k}(\overline{\Omega})}^{E}}$ is metrisable, $(\mathcal{C}^{k,\alpha}(\overline{\Omega},E),\|\cdot\|^{E},\tau_{|\cdot|_{\mathcal{C}^{k}(\overline{\Omega})}^{E}})$ is C-sequential by 
\cite[Proposition 5.7, p.~2681--2682]{kruse_meichnser_seifert2018}.

(b) By \cite[3.9 Example (vi), p.~11]{kruse2023b} 
$(\mathcal{C}^{k,\alpha}(\overline{\Omega}),\|\cdot\|,\tau_{|\cdot|_{\mathcal{C}^{k}(\overline{\Omega})}})$ 
is semi-Montel. Moreover, the observations 
$\tau_{\mathcal{N}_{k,\overline{\Omega}_{\operatorname{wd}}}}\leq\tau_{\mathcal{K}_{k,\overline{\Omega}_{\operatorname{wd}}}}\leq\tau_{|\cdot|_{\mathcal{C}^{k}(\overline{\Omega})}}$ on $\mathcal{C}^{k,\alpha}(\overline{\Omega})$ imply
that $\gamma(\|\cdot\|,\tau_{|\cdot|_{\mathcal{C}^{k}(\overline{\Omega})}})
=\gamma_{\operatorname{s}}(\|\cdot\|,\tau_{\mathcal{S}})$ for 
$\mathcal{S}\in\{\mathcal{N}_{k,\overline{\Omega}_{\operatorname{wd}}},\mathcal{K}_{k,\overline{\Omega}_{\operatorname{wd}}}\}$ by \prettyref{rem:pre_Saks_unit_ball_char} (c).

(c) We have $\mathcal{C}^{k,\alpha}(\overline{\Omega},E)
=\mathcal{C}^{k,\alpha}(\overline{\Omega},E)_{\sigma}$ by \cite[5.3.3 Corollary, p.~106]{kruse2023}. 
From parts (a) and (b), $\tau_{\operatorname{p}}\leq\tau_{\mathcal{K}_{k,\overline{\Omega}_{\operatorname{wd}}}}$ 
and \prettyref{prop:relation_weak_strong} (c) we deduce that the Saks space $(\mathcal{C}^{k,\alpha}(\overline{\Omega},E),\|\cdot\|^{E},\tau_{|\cdot|_{\mathcal{C}^{k}(\overline{\Omega})}^{E}})$ is complete.
\end{proof}

\begin{rem}
Let the assumptions of \prettyref{cor:hoelder_diff__mixed} (b) be fulfilled. 
Since the set $\{\beta\in\N^{d}_{0}\;|\;|\beta|=k\}$ is finite and 
$\tau_{\mathcal{N}_{\Lambda}}=\tau_{\mathcal{N}_{k,\overline{\Omega}_{\operatorname{wd}}}}$, 
we obtain analogously to the proof of \prettyref{thm:standard_saks_space} (d) that 
$\gamma(\|\cdot\|,\tau_{|\cdot|_{\mathcal{C}^{k}(\overline{\Omega})}})=\gamma(\|\cdot\|,\tau_{\mathcal{N}_{k,\overline{\Omega}_{\operatorname{wd}}}})$ is generated by the system of seminorms 
\[
\vertiii{f}_{(x_{n},y_{n},a_{n})_{n\in\N}}\coloneqq\sup_{\substack{n\in\N\\ \beta\in\N^{d}_{0}, |\beta|= k}}\Bigl(
|f|_{\mathcal{C}^{k}(\overline{\Omega})}+\frac{|(\partial^{\beta})f(x_{n})-(\partial^{\beta})f(y_{n})|}{|x_{n}-y_{n}|^{\alpha}}\Bigr)a_{n}, \quad f\in \mathcal{C}^{k,\alpha}(\overline{\Omega}),
\]
where $(x_{n},y_{n})_{n\in\N}$ is a sequence in $\overline{\Omega}_{\operatorname{wd}}$ and $(a_{n})_{n\in\N}\in c_{0}^{+}$.

Let $\mathcal{W}_{0}\coloneqq\mathcal{W}_{\operatorname{b},0}^{+}(\overline{\Omega}_{\operatorname{wd}})$ or 
$\mathcal{W}_{\operatorname{usc},0}^{+}(\overline{\Omega}_{\operatorname{wd}})$ or $\mathcal{C}_{0}^{+}(\overline{\Omega}_{\operatorname{wd}})$. 
Since $\{\beta\in\N^{d}_{0}\;|\;|\beta|=k\}$ is a finite set, 
$\tau_{\mathcal{K}_{\Lambda}}=\tau_{\mathcal{K}_{k,\overline{\Omega}_{\operatorname{wd}}}}$ and 
$\overline{\Omega}_{\operatorname{wd}}$ locally compact, 
we may also modify the system of seminorms in \prettyref{thm:standard_saks_space} (e) and obtain that 
$\gamma(\|\cdot\|,\tau_{|\cdot|_{\mathcal{C}^{k}(\overline{\Omega})}})=\gamma(\|\cdot\|,\tau_{\mathcal{K}_{k,\overline{\Omega}_{\operatorname{wd}}}})$ is generated by the system of seminorms
\[
|f|_{w}^{\sim}\coloneqq \sup_{\substack{x,y\in\Omega\\x\neq y}}
\sup_{\beta\in\N^{d}_{0}, |\beta|= k}\Bigl(|f|_{\mathcal{C}^{k}(\overline{\Omega})}+\frac{|\partial^{\beta}f(x)-\partial^{\beta}f(y)|}{|x-y|^{\alpha}}\Bigr)w(x,y), \quad f\in \mathcal{C}^{k,\alpha}(\overline{\Omega}),
\]
for $w\in\mathcal{W}_{0}$ (without the modification the weights $w$ depend on $\beta$ as well).
\end{rem}

\section{The dual space of \texorpdfstring{$(\FVE,\gamma_{\operatorname{s}}(\|\cdot\|^{E},\tau_{\mathcal{N}_{\Lambda}}^{E}))$}{Fv(Omega,E) w.r.t.~a submixed topology}}
\label{sect:dual_submixed}

In our closing section we give a characterisation of the dual space of the space $\FVE$ 
from \prettyref{defn:standard_space} w.r.t.~the submixed topology 
$\gamma_{\operatorname{s}}(\|\cdot\|^{E},\tau_{\mathcal{N}_{\Lambda}}^{E})$. We know from the preceding section 
that this submixed topology often coincides with the mixed topology, at least if $\Omega$ is discrete or $E=\K$. 
Our proof is an adaptation of the proof 
of the corresponding result \cite[Theorem 5.1, p.~652]{vargas2018} for the case $\FVE=\mathrm{Lip}_{0}(\Omega,E)$. 
For a normed space $(E,\|\cdot\|_{E})$ we denote by $E\oplus_{1} E$ the space $E\times E$ equipped with the norm 
$\|\cdot\|_{E\oplus_{1} E}$ given by $\|(x,y)\|_{E\oplus_{1} E}\coloneqq \|x\|_{E}+\|y\|_{E}$ for $x,y\in E$. 
By $\ell^{1}(\N,(E\oplus_{1} E)^{\ast})$ we denote the space of $(E\oplus_{1} E)^{\ast}$-valued sequences 
$y=(y_{n})_{n\in\N}$ such that $\|y\|_{1}\coloneqq \sum_{n=1}^{\infty}\|y_{n}\|_{(E\oplus_{1} E)^{\ast}}<\infty$. 

\begin{thm}\label{thm:dual_submixed_finite}
Let $\Omega$ and $\Lambda$ be non-empty sets, $v\colon\Lambda\to (0,\infty)$, $(E,\|\cdot\|_{E})$ a normed space over $\K$, 
$\GE$ a linear subspace of $E^{\Omega}$, $q^{E}\colon\GE\to[0,\infty)$ a seminorm, $T^{E}\colon \GE\to E^{\Lambda}$ a linear map.  
Suppose that $(\FVE,\|\cdot\|^{E})$ is normed, $q^{E}= \|T_{0}^{E}(\cdot)\|_{E}$ for some linear map $T_{0}^{E}\colon\GE\to E$ 
and $f'\colon\FVE\to\K$. Then $f'\in (\FVE,\gamma_{\operatorname{s}}(\|\cdot\|^{E},\tau_{\mathcal{N}_{\Lambda}}^{E}))'$ if and only if 
there are $(\lambda_{n})_{n\in\N}\in\ell^{1}(\N,(E\oplus_{1} E)^{\ast})$ and $(x_{n})_{n\in\N}$ in $\Lambda$ such that 
\[
f'(f)=\sum_{n=1}^{\infty}\lambda_{n}\bigl(T_{0}^{E}(f),T^{E}(f)(x_{n})v(x_{n})\bigr),\quad f\in\FVE .
\]
\end{thm}
\begin{proof}
$\Leftarrow$ Let there be sequences $(\lambda_{n})_{n\in\N}\in\ell^{1}(\N,(E\oplus_{1} E)^{\ast})$ and $(x_{n})_{n\in\N}$ 
in $\Lambda$ such that 
\[
f'(f)=\sum_{n=1}^{\infty}\lambda_{n}\bigl(T_{0}^{E}(f),T^{E}(f)(x_{n})v(x_{n})\bigr) 
\]
for all $\FVE$. Then $f'$ is linear. Since $\sum_{n=1}^{\infty}\|\lambda_{n}\|_{(E\oplus_{1} E)^{\ast}}<\infty$, 
there is a positive sequence $(\mu_{n})_{n\in\N}$ such that $\lim_{n\to\infty}\mu_{n}=\infty$ and 
$C\coloneqq\sum_{n=1}^{\infty}\mu_{n}\|\lambda_{n}\|_{(E\oplus_{1} E)^{\ast}}<\infty$ 
by \cite[Chap.~IX, \S 39, Theorem of Dini, p.~293]{knopp1951}. It follows that 
\begin{align*}
     |f'(f)|
&\leq\sum_{n=1}^{\infty}\|\lambda_{n}\|_{(E\oplus_{1} E)^{\ast}}(\|T_{0}^{E}(f)\|_{E}+\|T^{E}(f)(x_{n})\|_{E}v(x_{n}))\\
&\leq C\sup_{n\in\N}(\|T_{0}^{E}(f)\|_{E}+\|T^{E}(f)(x_{n})\|_{E}v(x_{n}))\mu_{n}^{-1}
 = C\|f\|_{(x_{n},\mu_{n}^{-1})_{n\in\N}}^{E}
\end{align*}
for all $f\in\FVE$, implying $f'\in (\FVE,\gamma_{\operatorname{s}}(\|\cdot\|^{E},\tau_{\mathcal{N}_{\Lambda}}^{E}))'$ 
by \prettyref{thm:standard_saks_space} (d).

$\Rightarrow$ Let $f'\in (\FVE,\gamma_{\operatorname{s}}(\|\cdot\|^{E},\tau_{\mathcal{N}_{\Lambda}}^{E}))'$. Then there are a sequence 
$(x_{n})_{n\in\N}$ in $\Lambda$, $(a_{n})_{n\in\N}\in c_{0}^{+}$ and $C\geq 0$ such that 
\begin{equation}\label{eq:dual}
|f'(f)|\leq C\|f\|_{(x_{n},a_{n})_{n\in\N}}^{E}=\sup_{n\in\N}(\|T_{0}^{E}(f)\|_{E}+\|T^{E}(f)(x_{n})\|_{E}v(x_{n}))\widetilde{a}_{n}
\end{equation}
for all $f\in\FVE$ by \prettyref{thm:standard_saks_space} (d) where $\widetilde{a}_{n}\coloneqq Ca_{n}$ for all $n\in\N$. 
Let $c_{0}(\N,E\oplus_{1} E)$ denote the space of $(E\oplus_{1} E)$-valued null-sequences on $\N$. We define the linear subspace 
\[
X\coloneqq \{(T_{0}^{E}(f)\widetilde{a}_{n},T^{E}(f)(x_{n})v(x_{n})\widetilde{a}_{n})_{n\in\N}\;|\;f\in\FVE\}
\]
of $c_{0}(\N,E\oplus_{1} E)$ and the functional $g^{\ast}\colon X\to \K$ given by 
\[
g^{\ast}\bigl((T_{0}^{E}(f)\widetilde{a}_{n},T^{E}(f)(x_{n})v(x_{n})\widetilde{a}_{n})_{n\in\N}\bigr)\coloneqq f'(f).
\]
The functional $g^{\ast}$ is well-defined and linear by \eqref{eq:dual} combined with the linearity of $T_{0}^{E}$ and $T^{E}$. The estimate 
\begin{flalign*}
&\hspace{0.37cm}\bigl|g^{\ast}\bigl((T_{0}^{E}(f)\widetilde{a}_{n},T^{E}(f)(x_{n})v(x_{n})\widetilde{a}_{n})_{n\in\N}\bigr)\bigr|\\
&=|f'(f)|
 \underset{\eqref{eq:dual}}{\leq}\sup_{n\in\N}(\|T_{0}^{E}(f)\|_{E}+\|T^{E}(f)(x_{n})\|_{E}v(x_{n}))\widetilde{a}_{n}\\
&=\sup_{n\in\N}\|(T_{0}^{E}(f)\widetilde{a}_{n},T^{E}(f)(x_{n})v(x_{n})\widetilde{a}_{n})\|_{E\oplus_{1} E}\\
&=\|(T_{0}^{E}(f)\widetilde{a}_{n},T^{E}(f)(x_{n})v(x_{n})\widetilde{a}_{n})_{n\in\N}\|_{\infty}
\end{flalign*}
for all $f\in\FVE$ implies that $g^{\ast}$ is continuous on $(X,{\|\cdot\|_{\infty}}_{\mid X})$ where $\|\cdot\|_{\infty}$ denotes 
the supremum norm on $c_{0}(\N,E\oplus_{1} E)$. Due to the Hahn--Banach theorem there exists an 
extension $\widehat{g}^{\,\ast}\in (c_{0}(\N,E\oplus_{1} E),\|\cdot\|_{\infty})'$ of $g^{\ast}$. 
Since the map 
\[
\Theta\colon(\ell^{1}(\N,(E\oplus_{1} E)^{\ast}),\|\cdot\|_{1})\to (c_{0}(\N,E\oplus_{1} E),\|\cdot\|_{\infty})',\; 
\Theta(y)(z)\coloneqq \sum_{n=1}^{\infty}y_{n}(z_{n}),
\]
is an isometric isomorphism, there is $(\kappa_{n})_{n\in\N}\in\ell^{1}(\N,(E\oplus_{1} E)^{\ast})$ such that 
$\widehat{g}^{\,\ast}(z)=\sum_{n=1}^{\infty}\kappa_{n}(z_{n})$ for all $z\in c_{0}(\N,E\oplus_{1} E)$. 
We set $\lambda_{n}\coloneqq \kappa_{n}\widetilde{a}_{n}$ for all $n\in\N$ and note that 
\[
    \sum_{n=1}^{\infty}\|\lambda_{n}\|_{(E\oplus_{1} E)^{\ast}}
\leq \|(\kappa_{n})_{n\in\N}\|_{1}(\sup_{m\in\N}|\widetilde{a}_{m}|)<\infty,
\]
implying $(\lambda_{n})_{n\in\N}\in\ell^{1}(\N,(E\oplus_{1} E)^{\ast})$. Finally, we conclude that 
\begin{align*}
  f'(f)
&=\widehat{g}^{\,\ast}\bigl((T_{0}^{E}(f)\widetilde{a}_{n},T^{E}(f)(x_{n})v(x_{n})\widetilde{a}_{n})_{n\in\N}\bigr)\\
&=\sum_{n=1}^{\infty}\lambda_{n}\bigl(T_{0}^{E}(f),T^{E}(f)(x_{n})v(x_{n})\bigr)
\end{align*} 
for all $f\in\FVE$. 
\end{proof}

\begin{rem}
If $q^{E}=0$, then we may take $(\lambda_{n})_{n\in\N}\in\ell^{1}(\N,(\{0\}\oplus_{1} E)^{\ast})=\ell^{1}(\N,E^{\ast})$ 
in \prettyref{thm:dual_submixed_finite}. 
\end{rem}

We observe that \prettyref{thm:dual_submixed_finite} is applicable to the spaces $\FVE=\mathcal{C}v(\Omega,E)$, 
$\mathcal{H}v(\Omega,E)$, $\mathcal{C}_{P}v(\Omega,E)$ and $\mathrm{Lip}_{0}(\Omega,E)$ with $T^{E}_{0}(f)\coloneqq 0$ for 
$f\in\FVE$, implying $q^{E}=0$, and to the space $\FVE=\mathcal{B}v(\D,E)$ with $T^{E}_{0}(f)\coloneqq f(0)$ 
for $f\in\mathcal{B}v(\D,E)$.

\bibliography{biblio_saks_mixed_vector_valued}
\bibliographystyle{plainnat}

\end{document}